\RequirePackage[l2tabu, orthodox]{nag} 
\documentclass[a4paper, 11pt, pdftex]{scrartcl}

\newcommand{\theauthor}{Geir Bogfjellmo and H\aa kon Marthinsen}
\newcommand{\thetitle}{High order symplectic partitioned Lie group methods}
 \newcommand{\thesubject}{Lie groups;mechanics;symplecticity;order conditions}
 \newcommand{\thekeywords}{Symplectic integrators; Lie groups; Order theory}
 \newcommand{\theamssubjectclassification}{Primary 65P10; Secondary 37M15, 70G65, 70G75, 70HXX}

\usepackage{fixltx2e} 
\ObsoleteEnv{displaymath}{LaTeX's equation*} 

\usepackage[utf8]{inputenc} 
\usepackage[T1]{fontenc} 

\usepackage[UKenglish]{babel} 
\usepackage[strict]{csquotes}

\usepackage[sortcites,backend=biber,isbn=false]{biblatex}
\defbibheading{bibliography}[\refname]{\section*{#1}}

\usepackage{microtype} 
\usepackage{amsmath, amsfonts, amssymb} 

\usepackage{amsthm} 
\makeatletter
\renewenvironment{proof}[1][\proofname]{\par
  \pushQED{\qed}%
  \normalfont \topsep6\p@\@plus6\p@\relax
  \trivlist
  \item[\hskip\labelsep
        \sffamily\bfseries
    #1\@addpunct{.}]\ignorespaces
}{%
  \popQED\endtrivlist\@endpefalse
}
\makeatother
\usepackage{mathtools} 
\mathtoolsset{mathic=true} 

\usepackage[final]{graphicx} 
\usepackage[svgnames]{xcolor} 
\usepackage{pgfplots} 

\usepackage{tikz} 
\usetikzlibrary{matrix,arrows}

\usepackage{subfig} 

\usepackage{url} 
\usepackage{datetime} 
\usepackage{todonotes} 

\usepackage{fourier}

\DeclareMathSymbol{\partial}{\mathord}{otherletters}{64}

\makeatletter
\def\resetMathstrut@{%
  \setbox\z@\hbox{%
    \mathchardef\@tempa\mathcode`\(\relax
    \def\@tempb##1"##2##3{\the\textfont"##3\char"}%
    \expandafter\@tempb\meaning\@tempa \relax
  }%
  \ht\Mathstrutbox@1.2\ht\z@ \dp\Mathstrutbox@1.2\dp\z@
}
\makeatother

\newdimen\arrowsize
\newdimen\unit
\pgfarrowsdeclare{fourier}{fourier}{\pgfarrowsrightextend{\pgflinewidth}}
{
    \unit=\pgflinewidth
    \divide\unit by 12
    \pgfpathmoveto{\pgfpoint{12\unit}{3\unit}}
    \pgfpathcurveto{\pgfpoint{0\unit}{7\unit}}{\pgfpoint{-18\unit}{19\unit}}{\pgfpoint{-26\unit}{26\unit}}
    \pgfpathlineto{\pgfpoint{-30\unit}{21\unit}}
    \pgfpathcurveto{\pgfpoint{-25\unit}{16\unit}}{\pgfpoint{-20\unit}{9\unit}}{\pgfpoint{-18\unit}{6\unit}}
    \pgfpathlineto{\pgfpoint{-18\unit}{-6\unit}}
    \pgfpathcurveto{\pgfpoint{-20\unit}{-9\unit}}{\pgfpoint{-25\unit}{-16\unit}}{\pgfpoint{-30\unit}{-21\unit}}
    \pgfpathlineto{\pgfpoint{-26\unit}{-26\unit}}
    \pgfpathcurveto{\pgfpoint{-18\unit}{-19\unit}}{\pgfpoint{0\unit}{-7\unit}}{\pgfpoint{12\unit}{-3\unit}}
    \pgfpathclose
    \pgfusepathqfill
}
\tikzstyle{cdp}=[auto, semithick, -fourier, font=\scriptsize]
\tikzstyle{cdp2}=[auto, semithick, serif cm-fourier, font=\scriptsize]
\tikzstyle{cdm}=[matrix of math nodes, row sep=3em, column sep=3em, text height=1.5ex, text depth=0.25ex]

\usepackage[scale=0.86]{tgadventor} 
\usepackage[scaled=0.84]{beramono} 

\usepackage{bm} 

\newcommand{\theoremname}{Theorem}
\newcommand{\propositionname}{Proposition}
\newcommand{\corollaryname}{Corollary}
\newcommand{\lemmaname}{Lemma}
\newcommand{\definitionname}{Definition}
\newcommand{\remarkname}{Remark}
\newcommand{\exercisename}{Exercise}
\newcommand{\examplename}{Example}

\usepackage[plain]{fancyref} 
\newcommand*{\fancyrefthmlabelprefix}{thm}
\newcommand*{\fancyrefprnlabelprefix}{prn}
\newcommand*{\fancyrefcorlabelprefix}{cor}
\newcommand*{\fancyreflemlabelprefix}{lem}
\newcommand*{\fancyrefdeflabelprefix}{def}
\newcommand*{\fancyrefremlabelprefix}{rem}
\newcommand*{\fancyrefexelabelprefix}{exe}
\newcommand*{\fancyrefexalabelprefix}{exa}

\fancyrefaddcaptions{english}{%
    \newcommand*{\Frefthmname}{\theoremname}%
    \newcommand*{\Frefprnname}{\propositionname}%
    \newcommand*{\Frefcorname}{\corollaryname}%
    \newcommand*{\Freflemname}{\lemmaname}%
    \newcommand*{\Frefdefname}{\definitionname}%
    \newcommand*{\Frefremname}{\remarkname}%
    \newcommand*{\Frefexename}{\exercisename}%
    \newcommand*{\Frefexaname}{\examplename}%
    \newcommand*{\frefthmname}{\MakeLowercase{\Frefthmname}}%
    \newcommand*{\frefprnname}{\MakeLowercase{\Frefprnname}}%
    \newcommand*{\frefcorname}{\MakeLowercase{\Frefcorname}}%
    \newcommand*{\freflemname}{\MakeLowercase{\Freflemname}}%
    \newcommand*{\frefdefname}{\MakeLowercase{\Frefdefname}}%
    \newcommand*{\frefremname}{\MakeLowercase{\Frefremname}}%
    \newcommand*{\frefexename}{\MakeLowercase{\Frefexename}}%
    \newcommand*{\frefexaname}{\MakeLowercase{\Frefexaname}}%
}

\frefformat{plain}{\fancyrefthmlabelprefix}{%
  \frefthmname\fancyrefdefaultspacing#1%
}
\frefformat{plain}{\fancyrefprnlabelprefix}{%
  \frefprnname\fancyrefdefaultspacing#1%
}
\frefformat{plain}{\fancyrefcorlabelprefix}{%
  \frefcorname\fancyrefdefaultspacing#1%
}
\frefformat{plain}{\fancyreflemlabelprefix}{%
  \freflemname\fancyrefdefaultspacing#1%
}
\frefformat{plain}{\fancyrefdeflabelprefix}{%
  \frefdefname\fancyrefdefaultspacing#1%
}
\frefformat{plain}{\fancyrefremlabelprefix}{%
  \frefremname\fancyrefdefaultspacing#1%
}
\frefformat{plain}{\fancyrefexelabelprefix}{%
  \frefexename\fancyrefdefaultspacing#1%
}
\frefformat{plain}{\fancyrefexalabelprefix}{%
  \frefexaname\fancyrefdefaultspacing#1%
}

\Frefformat{plain}{\fancyrefthmlabelprefix}{%
  \Frefthmname\fancyrefdefaultspacing#1%
}
\Frefformat{plain}{\fancyrefprnlabelprefix}{%
  \Frefprnname\fancyrefdefaultspacing#1%
}
\Frefformat{plain}{\fancyrefcorlabelprefix}{%
  \Frefcorname\fancyrefdefaultspacing#1%
}
\Frefformat{plain}{\fancyreflemlabelprefix}{%
  \Freflemname\fancyrefdefaultspacing#1%
}
\Frefformat{plain}{\fancyrefdeflabelprefix}{%
  \Frefdefname\fancyrefdefaultspacing#1%
}
\Frefformat{plain}{\fancyrefremlabelprefix}{%
  \Frefremname\fancyrefdefaultspacing#1%
}
\Frefformat{plain}{\fancyrefexelabelprefix}{%
  \Frefexename\fancyrefdefaultspacing#1%
}
\Frefformat{plain}{\fancyrefexalabelprefix}{%
  \Frefexaname\fancyrefdefaultspacing#1%
}

\usepackage[unicode]{hyperref} 
\hypersetup{pdfauthor=\theauthor,
    pdftitle=\thetitle,
    pdfsubject=\thesubject,
    pdfkeywords=\thekeywords,
    colorlinks,
    citecolor=Green,
    urlcolor=Blue,
    linkcolor=Red}

\DeclareMathOperator{\ad}{ad}
\DeclareMathOperator{\Ad}{Ad}

\DeclareMathOperator{\LieSO}{SO}

\DeclareMathOperator{\Lieso}{\mathfrak{so}}

\DeclareMathOperator{\dexp}{dexp}

\DeclareMathOperator{\diag}{diag}
\DeclareMathOperator{\OO}{\mathcal{O}}

\newcommand{\ee}{\mathrm{e}}
\newcommand{\dd}{\mathrm{d}}

\newcommand{\id}{\mathrm{id}}

\newcommand{\II}{\bm{I}}
\newcommand{\RR}{\mathbb{R}}

\newcommand{\DD}{\mathbf{D}}

\newcommand{\g}{\mathfrak{g}}

\newcommand{\norm}[1]{\lVert #1 \rVert}

\newcommand{\defeq}{\coloneqq}
\newcommand{\diff}[2]{\frac{\dd #1}{\dd #2}}
\newcommand{\pdiff}[2]{\frac{\partial #1}{\partial #2}}
\newcommand{\from}{\mathpunct{:}}

\newcommand{\trans}{^\mathrm{T}}
\newcommand{\coT}{T^{*}\!}

\newcommand{\hot}{\text{h.o.t.}}

\begin{document}

\newtheoremstyle{thmstyle} 
	{\topsep}              
	{\topsep}              
	{\itshape}             
	{}                     
	{\bfseries\sffamily}   
	{.}                    
	{.5em}                 
	{\thmname{#1}\thmnumber{ #2}\thmnote{ (#3)}} 
\theoremstyle{thmstyle}
\newtheorem{theorem}{\theoremname}[section]
\newtheorem{proposition}[theorem]{\propositionname}
\newtheorem{corollary}[theorem]{\corollaryname}
\newtheorem{lemma}[theorem]{\lemmaname}

\newtheoremstyle{defstyle}
	{\topsep}
	{\topsep}
	{}
	{}
	{\bfseries\sffamily}
	{.}
	{.5em}
	{\thmname{#1}\thmnumber{ #2}\thmnote{ (#3)}}
\theoremstyle{defstyle}
\newtheorem{definition}[theorem]{\definitionname}
\newtheorem{remark}[theorem]{\remarkname}
\newtheorem{example}[theorem]{\examplename}
\newtheorem{exercise}{\exercisename}[section]


\title{\thetitle}
\author{Geir Bogfjellmo\footnote{ \href{mailto:bogfjell@math.ntnu.no}{\nolinkurl{bogfjell@math.ntnu.no}}, Phone +47 73 59 17 53, Fax +47 73 59 35 24 (Corresponding author)}\, and H\aa kon Marthinsen\footnote{\href{mailto:hakonm@math.ntnu.no}{\nolinkurl{hakonm@math.ntnu.no}}} \\ \small{Department of Mathematical Sciences, NTNU, N--7491 Trondheim, Norway}}
\date{\small{23 April 2014}}
\maketitle



\begin{abstract}
        \noindent In this article, a unified approach to obtain symplectic integrators on $\coT G$ from Lie group integrators on a Lie group~$G$ is presented.
        The approach is worked out in detail for symplectic integrators based on Runge--Kutta--Munthe-Kaas methods and Crouch--Grossman methods.
        These methods can be interpreted as symplectic partitioned Runge--Kutta methods extended to the Lie group setting in two different ways.
        In both cases, we show that it is possible to obtain symplectic integrators of arbitrarily high order by this approach.
\end{abstract}

\noindent{\textbf{Keywords:}} \thekeywords

\noindent{\textbf{Mathematics Subject Classification (2010):}} \theamssubjectclassification

%
%
%

\section{Introduction}


\subsection{Motivation and background}
In general, an ordinary differential equation (ODE) can be described by a vector field on a smooth manifold where solutions of the ODE are integral curves of the vector field.
Numerical approximation of solutions of ODEs is an old field of study, and a plethora of methods for obtaining numerical solutions exist.
However, most of these methods assume that the manifold is Euclidean space.
If the manifold is not Euclidean space, it is possible to embed the manifold in Euclidean space, and extend the vector field on the manifold to a vector field in Euclidean space such that the integral curves are ensured to remain in the image of the embedding.
A standard numerical algorithm (e.g.\ a Runge--Kutta method) will in general result in discrete points which do not lie in the image of the embedding. 
An improvement of this approach is to use projection methods to obtain solutions on the manifold.
These approaches, though simple, suffer from the problem that the numerical solutions depend on the particular choice of embedding, and on the particular extension of the vector field.

One aspect of geometric numerical integration is to exploit structure on the manifold to define numerical methods that are intrinsic to the manifold (i.e.\ do not depend on a particular embedding).
This structure can for instance be that of a Lie group acting on the manifold.
The action of a Lie group $G$ on a manifold $M$ is a smooth mapping $\Psi \from G \times M \to M$ which respects the group structure on $G$.
If the action is transitive, then the derivative with respect to the first component of $\Psi$ at the group identity $e$ is a surjective vector bundle morphism $\g \times M \to TM$.
Any vector field $X$ on $M$ can then be lifted (possibly in a non-unique manner) to a section of the vector bundle $\g \times M$.
The combination of this lifting and standard charts $\g \to G$, form the basis of several classes of Lie group methods. 
Among them are the Crouch--Grossman (CG) methods \cite{crouch93} and the Runge--Kutta--Munthe-Kaas (RKMK) methods \cite{munthekaas99}.
For a more detailed discussion of Lie group methods, we refer to the survey article by Iserles et al.\ \cite{iserles00} and the references therein.

Another aspect of geometric numerical integration is symplecticity of numerical integrators.
Many important problems from physics can be formulated as Hamiltonian ODEs on cotangent bundles over manifolds.
The flow maps of these ODEs are symplectic, that is, they preserve the canonical two-form on the cotangent bundle.
For Hamiltonian ODEs, it is beneficial to use symplectic integrators, due to the near-preservation of energy and excellent long-term behaviour of the numerical solutions \cite[Chapter~VI]{hairer06}.
Hamilton's principle states that the solution of a Lagrangian (in many cases also Hamiltonian) system moves along a path which extremizes the action integral $S = \int_0^T L\bigl(q(t), \dot{q}(t)\bigr) \, \dd t $ among all paths~$q$ with fixed end points.

One technique for deriving symplectic methods is based on the notion of discretizing Hamilton's principle, that is, replacing the action integral with a discrete action sum, and extremizing over all discrete paths or sequences of points~$q_0, q_1, \dotsc, q_N$ with fixed end points.
These methods are known as \emph{variational methods} or \emph{variational integrators}.
Variational methods are guaranteed to be symplectic since the terms $L_h(q_{k - 1}, q_k)$ of the discrete action sum can be interpreted as generating functions (of the first type) for the numerical flow map.
Variational methods have been studied by numerous authors, we refer to the review article by Marsden and West \cite{marsden01} or the more recent encyclopedia article by Leok \cite{leok11} and the references therein for more information about variational methods.

Standard Lie group methods, like RKMK methods or CG methods, give numerical solutions that evolve on the same manifolds as the exact solutions. 
The question of the existence of symplectic methods of formats similar to the ones considered by Crouch and Grossman or by Munthe-Kaas has been a topic of interest for several years.

On $\RR^n$ there is a unified way to extend Runge--Kutta (RK) methods on $\RR^n$ to symplectic methods on $\coT \RR^n$ \cite[Section~VI.6.3]{hairer06}, i.e.\ symplectic partitioned RK (SPRK) methods.
Our goal with this article is to construct and study symplectic methods of arbitrarily high order that are extended from Lie group methods, i.e.\ high-order symplectic Lie group integrators.
We have focused on the case where $M=G$ and the action is simply multiplication in the Lie group.
In the case where $M\neq G$, isotropy complicates matters.
The technique of extremizing a discrete action sum still yields symplectic mappings in the case $M\neq G$, but the presence of isotropy complicates the analysis of these integrators.
The details of the isotropy case will hopefully be addressed in a later article.

The idea of constructing variational methods from Lie group methods has previously been considered by several authors.
Bou-Rabee and Marsden proposed in 2009 to base variational methods on RKMK methods \cite{bourabee09}, and present a class of methods of first and second order.
Methods of a similar type are also considered in the survey article by Celledoni, Marthinsen and Owren \cite{celledoni13}.
In the present article, this idea is pursued further to obtain methods of arbitrarily high order.

It is known to the authors that a different, but related approach to variational Lie group methods has been studied by Leok and collaborators.
Their approach is based on approximating the curve~$q$ in a finite-dimensional function space, resulting in Galerkin Lie group variational integrators.
The idea appears already in Leok's doctoral thesis \cite[Section 5.3]{leok04}, and also in other articles by Leok and co-authors.
A more detailed study of this approach can be found in an article by Hall and Leok~\cite{hall14}.

The rest of Section~1 is an introduction to ODEs on a Lie group~$G$, the Hamilton--Pon\-try\-a\-gin~(HP) principle and the equivalent HP~equations, and variational integrators in general.
Section~2 begins by introducing a group structure on $\coT G$, or equivalently, on $G \times \g^{*}$ and a function~$f \from  G \times \g^{*} \to \g \times \g^{*}$ which together fully describe ODEs on $\coT G$.
Next, we introduce the general format for our integrators.
In Section~3, we first show that a subclass of our integrators that have been studied before~\cite{celledoni13,bourabee09} can not obtain higher than second order on general Hamiltonian problems.
We then show that our integrators can not obtain higher order than the underlying Lie group integrators.
In Section~4, we present two classes of higher order integrators which are based on RKMK integrators and CG methods respectively.
In Section~5, we show that both classes of methods from Section~4 can obtain arbitrarily high order, and we present general order conditions for the methods based on RKMK~integrators, and conditions for order 1--3 for the CG-based~integrators.
We test the two classes of methods numerically in Section~6, and show that they both achieve the correct order and that they both have small energy errors over long time.
Finally, in Section~7 we conclude and mention some possible topics for further work.

\subsection{ODEs on Lie groups}
Let $G$ be a finite-dimensional Lie group and $\g$ its associated Lie algebra.
We denote right-multiplication with $g \in G$ as $R_g$ and left-multiplication with $g$ as $L_g$.
We use dot notation to denote translation in the tangent bundle, i.e.
\begin{gather*}
	g \cdot v = T L_g v, \qquad v \cdot g = T R_g v, \qquad v \in TG, \\
	\intertext{and in the cotangent bundle}
	g \cdot p = \coT L_{g^{-1}} p, \qquad p \cdot g = \coT R_{g^{-1}} p, \qquad p \in \coT G.
\end{gather*}
We also need the notation $\Ad_g \defeq TL_g \circ TR_{g^{-1}}$.
All autonomous ODEs on $G$ can be written as
\begin{equation}
	\label{eq:ODE-on-Lie-group}
	\dot g = f(g) \cdot g, \qquad g(0) = g_0,
\end{equation}
where $g$ is a curve in $G$ and the map~$f \from G \to \g$ is determined uniquely by the vector field.
We can solve this kind of equation numerically using Lie group methods \cite{iserles00}.
Here we have chosen the \emph{right-trivialized} form of this equation. 
We could also have used the left-trivialized form~$\dot{g} = g\cdot f(g)$ which would have resulted in only minor changes to the formulae presented later in the article.

Since we are interested in solving Hamiltonian ODEs using Lie group methods, we need a group structure on the cotangent bundle of $G$, as well as the map $f$ that corresponds to this type of ODEs.

\subsection{Hamilton--Pontryagin mechanics}
Lagrangian mechanics on $G$ is formulated in terms of a Lagrangian~$L \from TG \to \RR$.
Hamilton's principle states that the dynamics is given by the curve~$q \from \RR \to G$ that extremizes the \emph{action integral}
\[
	S_\mathrm{H} = \int_0^T L(q, \dot q) \, \dd t,
\]
where the endpoints $q(0)$ and $q(T)$ are kept fixed.
In \cite[Theorem~3.4]{bourabee09} it was shown that this is equivalent to the Hamilton--Pontryagin (HP) principle, which states that the dynamics is given by extremizing
\[
	S_\mathrm{HP} = \int_0^T \left(L(q, v) + \langle p, \dot q - v \rangle\right) \, \dd t,
\]
where $v \in T_q G$, $p \in T^{*}_q G$ are varied arbitrarily, and the endpoints of $q$ are kept fixed.
Here, we denote the natural pairing of covectors and vectors by $\langle \cdot, \cdot \rangle$.
This action integral leads to dynamics formulated on $\coT G$.

To simplify further calculations, it is convenient to right-trivialize $\coT G$ to $G \times \g^{*}$ via the map~$(q,p_q) \mapsto (q, p_q \cdot q^{-1})$.
Letting $\ell(q, \xi) \defeq L(q, \xi \cdot q)$, $\xi \in \g$, it is easy to show that the HP principle is equivalent to the \emph{right-trivialized HP principle}, which has action integral
\[
	S = \int_0^T \left(\ell(q, \xi) + \langle \mu, \dot q \cdot q^{-1} - \xi \rangle\right) \, \dd t,
\]
where $\xi \from \RR \to \g$ and $\mu \from \RR \to \g^{*}$ are varied arbitrarily, and the endpoints of $q$ are kept fixed.
Taking the variation of $S$, we arrive at the \emph{right-trivialized HP equations}
\begin{equation}
	\label{eq:right-triv-HP}
	\begin{aligned}
		\dot q &= \xi \cdot q, \\
		\dot \mu &= -\ad_\xi^{*} \mu + \bigl(\DD_1 \ell(q, \xi)\bigr) \cdot q^{-1}, \\
		\mu &= \DD_2 \ell(q, \xi),
	\end{aligned}
\end{equation}
where $\ad_x$ is the derivative of $\Ad_{\exp(x)}$ with respect to $x$ at the origin, and $\DD_k \ell$ denotes the partial derivative of $\ell$ with respect to the $k$th variable, i.e.\ a one-form.
This is the ODE on $G \times \g^{*}$ that we need to solve.

\subsection{Variational integrators}

Variational integrators are constructed by discretizing an action integral and then performing extremization with fixed endpoints.
This procedure turns the action integral into an \emph{action sum.}
The \emph{discrete Lagrangian}~$L_h \from G \times G \to \RR$ is an approximation of the action integral over a small time step~$h$,
\[
	L_h(q_{k-1}, q_k) \approx \int_{(k-1)h}^{kh} L(q, \dot q) \, \dd t,
\]
where $q\from \RR \to G$ extremizes the action integral with $q(0)$ and $q(T)$ fixed.
Letting $N = T / h$, the action sum becomes
\[
	S_h = \sum_{k = 1}^N L_h(q_{k-1}, q_k).
\]
Extremizing $S_h$ while keeping $q_0$ and $q_N$ fixed gives us the \emph{discrete Euler--Lagrange equations}
\[
	\DD_1 L_h(q_k, q_{k+1}) + \DD_2 L_h(q_{k-1}, q_k) = 0, \quad 1 \leq k < N.
\]
The discrete Legendre transforms define the \emph{discrete conjugate momenta}
\begin{align*}
	p_k &\defeq \mu_k \cdot q_k \defeq -\DD_1 L_h(q_k, q_{k+1}), \\
	p_{k+1} &\defeq \mu_{k+1} \cdot q_{k+1} \defeq \DD_2 L_h(q_k, q_{k+1}).
\end{align*}
By demanding that these two definitions are consistent, we automatically satisfy the discrete Euler--Lagrange equations.
If we can solve the first equation for $q_{k+1}$, we can use the second one to calculate $\mu_{k+1}$, giving us the variational integrator~$(q_k, \mu_k) \mapsto (q_{k+1}, \mu_{k+1})$.

\section{From a Lie group method to a variational integrator on the cotangent bundle}
\label{sec:varprinciple}
\subsection{Group structure and Hamiltonian ODEs on \texorpdfstring{$G \times \g^{*}$}{G x g*}}

We want to numerically solve the right-trivialized HP equations~\Fref{eq:right-triv-HP}, which can be viewed as a vector field on $G \times \g^{*}$, or equivalently, as the ODE $\dot z = f(z) \cdot z$, where $z \in G \times \g^{*}$ and $f \from G \times \g^{*} \to \g \times \g^{*}$.
For this ODE to make sense, we must choose a group product on $G \times \g^{*}$.
We choose the magnetic extension of $G$, as described by Arnold and Khesin \cite[Section~I.10.B]{arnold98}.
As we will see, this group product makes the right-trivialized HP equations easily expressible as $\dot z = f(z) \cdot z$.\footnote{This group structure was used by Engø in \cite{engoe03} to construct partitioned Runge--Kutta--Munthe-Kaas methods on $\coT G$, without any special regard to symplecticity.}

The magnetic extension assigns the following group product to $\coT G$:
\begin{equation}
	(g, p_g) (h, p_h) \defeq (g h, p_g \cdot h + g \cdot p_h).
\label{eq:groupstructure}
\end{equation}
This group structure is an extension of the group structure on $G$ in the sense that the canonical projection $\coT G \to G$ is a homomorphism of Lie groups.

We note that
\begin{align*}
	(g, p_g) (h, p_h) &= (g h, p_g \cdot h + g \cdot p_h) \\
	&= \Bigl(g h, \bigl(p_g \cdot g^{-1} + \Ad_{g^{-1}}^{*} (p_h \cdot h^{-1})\bigr) \cdot g h\Bigr).
\end{align*}
Thus, letting $\mu = p_g \cdot g^{-1}$ and $\nu = p_h \cdot h^{-1}$, the right-trivialized version of \Fref{eq:groupstructure} is the product on $G \times \g^{*}$ defined by
\[
	(g, \mu) (h, \nu) \defeq \bigl(gh, \mu + \Ad_{g^{-1}}^{*} \nu\bigr).
\]
It can be shown that the Lie algebra associated to the Lie group $G \times \g^{*}$ is $\g \times \g^{*}$ equipped with the Lie bracket $[(\xi, \mu), (\eta, \nu)] = (\ad_\xi \eta, \ad^{*}_\eta \mu - \ad^{*}_\xi \nu)$.

We will also need an expression for $TR_z \zeta$ for $z = (q, \mu) \in G \times \g^{*}$ and $\zeta = (\eta, \nu) \in \g \times \g^{*}$:
\begin{align*}
	TR_z \zeta &= \diff{}{\epsilon} \bigl(\exp(\epsilon \eta), \epsilon \nu\bigr) (q, \mu)\biggr\rvert_{\epsilon = 0} \\
	&= \diff{}{\epsilon} \bigl(\exp(\epsilon \eta) q, \epsilon \nu + \Ad_{\exp(-\epsilon \eta)}^{*} \mu\bigr)\biggr\rvert_{\epsilon = 0} \\
	&= \bigl(\eta \cdot q, \nu - \ad_\eta^{*} \mu\bigr),
\end{align*}

We would now like to use this to write the right-trivialized HP equations \Fref{eq:right-triv-HP} in the form of \Fref{eq:ODE-on-Lie-group}, $\dot z = f(z) \cdot z = T R_z \circ f(z)$.
If the map~$f \from G \times \g^{*} \to \g \times \g^{*}$ satisfies
\begin{equation}
	\label{eq:f-function}
	f\bigl(q, \DD_2 \ell(q, \xi)\bigr) = \Bigl(\xi, \bigl(\DD_1\ell(q,\xi)\bigr) \cdot q^{-1}\Bigr)
\end{equation}
for all $(q, \xi) \in G \times \g$, we see that $\dot z = f(z) \cdot z$, which is exactly what we need.

In many cases, the map $(q,\xi) \mapsto \bigl(q,\DD_2 \ell(q, \xi)\bigr)$ is a diffeomorphism of manifolds.
If this holds, we say that the Lagrangian $\ell$ is \emph{regular}.
If $\ell$ is regular, the Lagrangian problem has an equivalent formulation as a Hamiltonian ODE on $\coT G$,
where 
\[\mathcal{H}\bigl(q, \DD_2 \ell(q, \xi)\bigr)= \langle \DD_2 \ell(q, \xi), \xi \rangle - \ell(q, \xi)\]
and
\begin{equation}
f(q, \mu) = \Bigl( \DD_2 \mathcal{H}(q, \mu), - \bigl(\DD_1 \mathcal{H}(q,\mu)\bigr) \cdot q^{-1} \Bigr).
\label{eq:hamfield}
\end{equation}
For Hamiltonians which arise in this manner, the map  $(q, \mu) \mapsto \bigl(q, \DD_2 \mathcal{H}(q, \mu)\bigr)$ is also a diffeomorphism of manifolds. In fact the map is the inverse of the one above.
Hamiltonians for which this hold are also called \emph{regular}.

%
%


\subsection{General format for our integrators} \label{sec:gen-format}
It is natural to consider discrete Lagrangians based on approximation of the action integral by quadrature.
The procedure adopted in the present article is inspired by the approach in \cite[Section~VI.6.3]{hairer06}, originally found in \cite{suris90}.
In this reference, the symplectic partitioned Runge--Kutta methods are derived by considering the discrete Lagrangian
\begin{equation} 
 L_h(q_0, q_1) = h \sum_{i=1}^s b_i L(Q_i, \dot{Q}_i)
 \label{eq:RKaction}
\end{equation}
where
\[Q_i=q_0+h\sum_{j=1}^s a_{ij} \dot{Q}_j,\]
and $b_i, a_{ij}$ are the coefficients of a Runge--Kutta method.
The $\dot{Q}_i$ are chosen to extremize the sum above under the constraint
\[q_1=q_0+h\sum_{i=1}^sb_i \dot{Q}_i.\]
As shown in \cite[Section~VI.6.3]{hairer06}, the resulting integrator is exactly the partitioned Runge--Kutta integrator where the position is integrated using the original coefficients $b_i,a_{ij}$, while the momentum is integrated by using the coefficents
$\hat{b}_i =b_i, \hat{a}_{ij} = b_j-b_ja_{ji}/b_i.$

In the following, we will generalize the approach used in \cite[Section~VI.6.3]{hairer06} to Lie groups.
Consider the discrete Lagrangian
\begin{equation*} 
	L_h(q_0, q_1) = \hat L_h(Q_1, \dotsc, Q_s, \xi_1, \dotsc, \xi_s) = h \sum_{i = 1}^s b_i \ell(Q_i, \xi_i),
\end{equation*}
where $b_i$ are non-zero quadrature weights, and the auxiliary variables $Q_1, \dotsc, Q_s, \xi_1, \dotsc, \xi_s$ are chosen to extremize $\hat L_h$ under the constraints
\begin{equation}
\begin{aligned}
	Y(Q_1, \dotsc, Q_s, \xi_1, \dotsc, \xi_s, q_0) - \log\bigl(q_1 q_0^{-1}\bigr) &= 0, \\
	X_i(Q_1, \dotsc, Q_s, \xi_1, \dotsc, \xi_s, q_0) - \log\bigl(Q_i q_0^{-1}\bigr) &= 0, \qquad i = 1, \dotsc, s.
\end{aligned}
\label{eq:varconstraints}
\end{equation}
The functions $Y$ and $X_i$ will typically arise from Lie group integrators, as we will see later on.
The formulation of the discrete Lagrangian is that of a constrained optimization problem.
As done in \cite[Section~VI.6.3]{hairer06}, we solve this by introducing Lagrange multipliers.
Let $\Lambda$ be the Lagrange multiplier corresponding to the constraint containing $Y$, and let $\lambda_i$ be the Lagrange multiplier corresponding to the equation containing $X_i$ for $i=1,\dots,s$.
To obtain a variational integrator, we extremize
\[
	\hat L_h - \bigl\langle \Lambda, Y - \log\bigl(q_1 q_0^{-1}\bigr) \bigr\rangle - \sum_{i = 1}^s \bigl\langle \lambda_i, X_i - \log\bigl(Q_i q_0^{-1}\bigr) \bigr\rangle,
\]
while keeping $q_0$ and $q_1$ fixed.
Varying this with respect to $\Lambda$, $\lambda_i$, $\xi_i$ and $Q_i$, we obtain the set of equations
\begin{equation}
\begin{aligned}
	q_1 &= \exp(Y) q_0, \\
	Q_i &= \exp(X_i) q_0,\\
	\pdiff{\hat L_h}{\xi_i} &= \left(\pdiff{Y}{\xi_i}\right)^{*} \Lambda + \sum_j \left(\pdiff{X_j}{\xi_i}\right)^{*} \lambda_j,\\
	\pdiff{\hat L_h}{Q_i} &= \left(\pdiff{Y}{Q_i}\right)^{*} \Lambda + \sum_j \left(\pdiff{X_j}{Q_i}\right)^{*} \lambda_j - \Bigl(\bigl(\dexp_{X_i}^{-1}\bigr)^{*} \lambda_i \Bigr) \cdot Q_i,
\end{aligned}
\label{eq:vareq1}
\end{equation}
for all $i = 1, \dotsc, s$.

To find the integrator based on the discrete Lagrangian~$L_h$, we need to evaluate the partial derivatives of $L_h$ with respect to $q_0$ and $q_1$.
In doing so, we consider $Q_1, \dotsc, Q_s, \xi_1, \dotsc, \xi_s$ as functions of $q_0$ and $q_1$ defined implicitly by \Fref{eq:varconstraints} and \Fref{eq:vareq1}.
The partial derivatives of $L_h$ are then
\begin{equation}
\begin{aligned}
\pdiff{L_h}{q_0} &= \sum_j \left(\pdiff{\hat{L}_h}{Q_j}\circ \pdiff{Q_j}{q_0} + \pdiff{\hat{L}_h}{\xi_j} \circ \pdiff{\xi_j}{q_0}\right), \\
\pdiff{L_h}{q_1} &= \sum_j \left(\pdiff{\hat{L}_h}{Q_j}\circ \pdiff{Q_j}{q_1} + \pdiff{\hat{L}_h}{\xi_j} \circ \pdiff{\xi_j}{q_1}\right).
\end{aligned}
\label{eq:partialcomp}
\end{equation}
The functions $Q_1, \dotsc, Q_s, \xi_1, \dotsc, \xi_s$ satisfy the  constraints \Fref{eq:varconstraints} for all $q_0$, $q_1$.
By differentiating the constraints we see that the identities
\begin{equation}
\begin{aligned}
 0&= \pdiff{Y}{q_0} + \sum_j \left(\pdiff{Y}{Q_j}\circ \pdiff{Q_j}{q_0} + \pdiff{Y}{\xi_j}\circ\pdiff{\xi_j}{q_0}\right) + \dexp_{-Y}^{-1} \circ TR_{q_0^{-1}},\\
 0&=\sum_j \left(\pdiff{Y}{Q_j}\circ \pdiff{Q_j}{q_1} + \pdiff{Y}{\xi_j}\circ \pdiff{\xi_j}{q_1}\right) - \dexp_{Y}^{-1} \circ TR_{q_1^{-1}},\\
 0&= \pdiff{X_i}{q_0} +\sum_j \left(\pdiff{X_i}{Q_j} \circ \pdiff{Q_j}{q_0} + \pdiff{X_i}{\xi_j}\circ \pdiff{\xi_j}{q_0} \right) +\dexp_{-X_i}^{-1} \circ TR_{q_0^{-1}} - \dexp_{X_i}^{-1}\circ TR_{Q_i^{-1}} \circ \pdiff{Q_i}{q_0},\\
0&= \sum_j \left(\pdiff{X_i}{Q_j} \circ \pdiff{Q_j}{q_1} + \pdiff{X_i}{\xi_j}\circ \pdiff{\xi_j}{q_1} \right) -\dexp_{X_i}^{-1} \circ TR_{Q_i^{-1}} \circ \pdiff{Q_i}{q_1}, \qquad i=1,\dotsc, s,
\end{aligned}
\label{eq:varidentities}
\end{equation}
all hold.

We combine the discrete Legendre transforms
\[
-\mu_0 \cdot q_0 = \pdiff{L_h}{q_0}, \qquad \mu_1 \cdot q_1 = \pdiff{L_h}{q_1},
\]
with \Fref{eq:vareq1} and \Fref{eq:partialcomp}, and simplify using \Fref{eq:varidentities} to obtain the equations
\begin{align*}
	\mu_0 &= \left( \left(\pdiff{Y}{q_0}\right)^{*} \Lambda + \sum_j \left(\pdiff{X_j}{q_0}\right)^{*} \lambda_j \right) \cdot q_0^{-1} + \bigl(\dexp_{-Y}^{-1}\bigr)^{*} \Lambda + \sum_j \bigl(\dexp_{-X_j}^{-1}\bigr)^{*} \lambda_j, \\
	\mu_1 &= \bigl(\dexp_Y^{-1}\bigr)^{*} \Lambda.
\end{align*}
Using the identity \Fref{eq:f-function}, we get
\[
	f\bigl(Q_i, \DD_2 \ell(Q_i, \xi_i)\bigr) = \Bigl(\xi_i, \bigl(\DD_1 \ell(Q_i, \xi_i)\bigr) \cdot Q_i^{-1}\Bigr),
\]
and defining $n_i, M_i \in \g^{*}$ by
\[
	\pdiff{\hat L_h}{Q_i} = h b_i \DD_1 \ell(Q_i, \xi_i) = h b_i n_i \cdot Q_i, \qquad \pdiff{\hat L_h}{\xi_i} = h b_i \DD_2 \ell(Q_i, \xi_i) = h b_i M_i,
\]
for $i=1,\dotsc, s,$ we get
\begin{align*}
	h b_i n_i &= \left( \left(\pdiff{Y}{Q_i}\right)^{*} \Lambda + \sum_j \left(\pdiff{X_j}{Q_i}\right)^{*} \lambda_j \right) \cdot Q_i^{-1} - \bigl(\dexp_{X_i}^{-1}\bigr)^{*} \lambda_i, \\
	h b_i M_i &= \left(\pdiff{Y}{\xi_i}\right)^{*} \Lambda + \sum_j \left(\pdiff{X_j}{\xi_i}\right)^{*} \lambda_j.
\end{align*}

Combining everything above, the variational integrator is defined by the set of equations
\begin{equation}
\begin{aligned}
	\mu_0 &= \left( \left(\pdiff{Y}{q_0}\right)^{*} \Lambda + \sum_j \left(\pdiff{X_j}{q_0}\right)^{*} \lambda_j \right) \cdot q_0^{-1} + \bigl(\dexp_{-Y}^{-1}\bigr)^{*} \Lambda + \sum_j \bigl(\dexp_{-X_j}^{-1}\bigr)^{*} \lambda_j, \\
	h b_i n_i &= \left( \left(\pdiff{Y}{Q_i}\right)^{*} \Lambda + \sum_j \left(\pdiff{X_j}{Q_i}\right)^{*} \lambda_j \right) \cdot Q_i^{-1} - \bigl(\dexp_{X_i}^{-1}\bigr)^{*} \lambda_i,\\
	h b_i M_i &= \left(\pdiff{Y}{\xi_i}\right)^{*} \Lambda + \sum_j \left(\pdiff{X_j}{\xi_i}\right)^{*} \lambda_j,\\
	(\xi_i, n_i) &= f(Q_i, M_i),\\
	Q_i &= \exp(X_i) q_0,\qquad \qquad \qquad i=1,\dotsc, s, \\
	q_1 &= \exp(Y) q_0, \\
	\mu_1 &= \bigl(\dexp_Y^{-1}\bigr)^{*} \Lambda.
\end{aligned}
\label{eq:varprinciple}
\end{equation}
Notice that we no longer involve the Lagrangian.
We only need to evaluate the vector field through the map~$f$.
This opens up the possibility of applying the method to degenerate Hamiltonian systems (or indeed to any ODE on $\coT G$).\footnote{Variational methods for degenerate Hamiltonian systems using Type~II generating functions have been proposed by Leok and Zhang \cite{leok11-1}.}

It should be noted that since the integrator can be formulated as a variational integrator on $G$,
the group structure chosen for $\coT G$ in \Fref{eq:groupstructure} is not consequential.
Indeed, the integrator is uniquely defined by \Fref{eq:varconstraints}, \Fref{eq:vareq1}, and \Fref{eq:partialcomp}, which do not depend on the introduced group structure on $\coT G$.
For any choice of group structure on $\coT G$ such that the canonical projection~$\coT G \to G$ is a homomorphism of Lie groups, there is an equivalent formulation of the integrator in \Fref{eq:varprinciple}.
Note that $f$ is defined via the group structure, and a change of group structure would lead to $f$ being changed as well.

%

\section{First and second order integrators}
\label{sec:secondorder}

In the article by Celledoni et al.~\cite{celledoni13}, a special case of variational integrators of the form introduced in the previous section was considered.
These integrators serve as an example of application of the formulae above.
In these methods, let $a_{ij}$ and $b_i$ be the coefficients of an $s$-stage Runge--Kutta method which satisfies $b_i \neq 0$ for all $i$.
Let the discrete Lagrangian be given by
\[L_h(q_0, q_1) = h\sum_{i=1}^s b_i \ell(Q_i, \xi_i),\]
and the constraints by \Fref{eq:varconstraints} and
\[
   Y=h\sum_{i=1}^s b_i \xi_i, \qquad X_i= h\sum_{j=1}^s a_{ij} \xi_j, \qquad i=1,\dotsc, s.
\]
We can see that for $i,j = 1,\dotsc, s,$
\[
\begin{aligned} 
\pdiff{Y}{q_0} &= 0,  & \pdiff{X_j}{q_0} &= 0, \\
\pdiff {Y}{Q_i} &= 0,  & \pdiff{X_j}{Q_i} &= 0, \\
\pdiff{Y}{\xi_i} &= hb_i, & \pdiff{X_j}{\xi_i} &= ha_{ji}.
\end{aligned}
\]
By inserting these into \Fref{eq:varprinciple}, we get the set of equations
\[
\begin{aligned}
	\mu_0 &= \bigl(\dexp_{-Y}^{-1}\bigr)^{*} \Lambda + \sum_j \bigl(\dexp_{-X_j}^{-1}\bigr)^{*} \lambda_j, \\
	h b_i n_i &= - \bigl(\dexp_{X_i}^{-1}\bigr)^{*} \lambda_i, \\
	h b_i M_i &= hb_i \Lambda + \sum_j ha_{ji} \lambda_j, \\
	(\xi_i, n_i) &= f(Q_i, M_i), \\
	Q_i &= \exp(X_i) q_0,  \qquad \qquad \qquad i=1,\dotsc, s,\\
	q_1 &= \exp(Y) q_0, \\
	\mu_1 &= \bigl(\dexp_Y^{-1}\bigr)^{*} \Lambda.
	\end{aligned}
\]
In these equations, $\Lambda$ and $\lambda_j$ can be eliminated, giving the integrator
\begin{equation}
\begin{aligned}
	b_i M_i &= b_i \dexp_{-Y}^{*}\Bigl( \mu_0 + h \sum_j b_j \Ad^{*}_{\exp(X_j)} n_j \Bigr) - h \sum_j b_j a_{ji} \dexp_{X_j}^{*} n_j, \\
	X_i &= h\sum_j a_{ij} \xi_j, \\
	Q_i &= \exp(X_i) q_0, \\
	(\xi_i, n_i) &= f(Q_i, M_i), \qquad \qquad \qquad i=1,\dotsc, s,\\
	Y &= h\sum_j b_j \xi_j,\\
	q_1 &= \exp(Y) q_0, \\
	\mu_1 &= \Ad^{*}_{\exp(-Y)}\Bigl(\mu_0 + h \sum_j b_j \Ad^{*}_{\exp(X_j)} n_j\Bigr).
\end{aligned}
\label{eq:JCP-int}
\end{equation}
Here we have used the identity $\dexp_{x}\circ \dexp^{-1}_{-x} = \Ad_{\exp(x)}$.
Equation~\Fref{eq:JCP-int} is equivalent to the method presented in~\cite[Section~5]{celledoni13}.

Methods of this form suffer from an order barrier.
They can not obtain higher accuracy than second order.
The proof, presented below, is closely related to a similar order barrier for commutator-free Lie algebra methods \cite{celledoni03-1}.

\begin{proposition}
The integrators of the format \Fref{eq:JCP-int} can not achieve higher than second order on general Hamiltonian differential equations.
\label{prn:JCP-order}
\end{proposition}
%
%
%
\begin{proof}
 The proof proceeds by applying the variational method \Fref{eq:JCP-int} to a particular class of regular Hamiltonian problems and a particular choice of starting values. 
 We show that in this case, the Lie group part of the solution has an error of at most second order, thus the variational method is at most of second order as well.
 
 Let a Hamiltonian on $G\times \g^{*}$ be given by $\mathcal{H}(q,\mu) = \langle \mu, v(q) \rangle+T(\mu)$, where $v\from G \to \g$ is smooth, but otherwise arbitrary, and $T\from \g^{*} \to \RR$ is a nondegenerate quadratic function of $\mu$.
 Using \Fref{eq:hamfield}, we find that the corresponding Hamiltonian vector field is
 \begin{equation}
 f(q, \mu) =  \left( v(q) + \frac{\dd T}{\dd \mu}, -\Biggl(\biggl(\pdiff{v}{q}\biggr)^{*} \mu\Biggr)\cdot q^{-1}\right),
 \label{eq:hamfieldspec}
 \end{equation}
and the differential equation is
 \[ \begin{aligned}
\dot{q} &= \biggl(v(q) + \frac{\dd T}{\dd \mu}\biggr)\cdot q, \\
\dot{\mu} &= -\Biggl(\biggl(\pdiff{v}{q}\biggr)^{*} \mu\Biggr)\cdot q^{-1}- \ad^*_{v(q)+ \frac{\dd T}{\dd \mu}}\mu.
\end{aligned} \]
We note that $\frac{\dd T}{\dd \mu}$ and $U(q)\mu\defeq-\Bigl(\bigl(\pdiff{v}{q}\bigr)^{*} \mu\Bigr)\cdot q^{-1}$ are both linear in $\mu$, in particular $\frac{\dd T}{\dd \mu}\bigr\rvert_{\mu=0}=0$.

A particular class of solutions to this ODE consists of those solutions which satisfy $\mu(t)=0$ for all $t$.
For these solutions $q(t)$ solves the ODE~$\dot{q} = v(q) \cdot q$.
We want to show that the numerical solution from \Fref{eq:JCP-int}, when applied to this problem, preserves the invariant $\mu=0$, and that the method reduces to a conventional Lie group method in this case.
If we apply \Fref{eq:JCP-int} to the Hamiltonian vector field \Fref{eq:hamfieldspec} and set $\mu_0=0$, we get, among others, the equation $n_i= U(Q_i)M_i$.
Inserting this into the first equation of \Fref{eq:JCP-int}, we get
\[
\begin{aligned}
	b_i M_i &= h b_i \dexp_{-Y}^{*} \sum_j b_j \Ad^{*}_{\exp(X_j)} U(Q_j)M_j - h \sum_j b_j a_{ji} \dexp_{X_j}^{*} U(Q_j)M_j, \qquad i=1,\dotsc, s. \\
\end{aligned}
\]
Clearly, this system of equations has $M_i=0$ for $i=1,\dotsc,s$ as a solution.
Additionally, if we assume that $Y$ and $X_j$ go to zero as $h$ goes to zero, $M_i=0$ is the only solution for small enough step-length~$h$.
Therefore, $n_i=U(Q_i)M_i=0$, and $\mu_1=0$.
The remaining equations of \Fref{eq:JCP-int} are
\[
 \begin{aligned}
  	X_i &= h\sum_j a_{ij} \xi_j, \\
	Q_i &= \exp(X_i) q_0, \\
	\xi_i &=v(Q_i)+ \frac{\dd T}{\dd \mu}\biggr\rvert_{\mu=M_i}=v(Q_i), \qquad \qquad \qquad i=1,\dotsc, s,\\
	Y &= h\sum_j b_j \xi_j,\\
	q_1 &= \exp(Y) q_0. \\
 \end{aligned}
\]

We recognise these equations as a commutator-free Lie group method with one exponential, or equivalently, an RKMK method with cut-off parameter 0, applied to the ODE $\dot{q} = v(q)\cdot q$.
As explained in \cite{owren06}, commutator-free methods with one exponential cannot satisfy the third order conditions, and the solution is at most second order accurate.

\end{proof}
By repeating the argument with any variational integrator of the form described in \Fref{eq:varprinciple}, we get a generalization.
\begin{proposition}
\label{prn:orderlimit}
 A variational integrator on $\coT G$ of the form \Fref{eq:varprinciple} based on a Lie group integrator can not achieve higher order than the underlying Lie group integrator.
\end{proposition}

We should note that first and second order methods of the format described in \Fref{eq:JCP-int} do exist.
Specifically, a method of that format is first order if $\sum_{i=1}^s b_i = 1$ and second order if, in addition, $\sum_{i,j=1}^s b_i a_{ij} = 1/2$.
The proof is a special case of \Fref{thm:orderRKMK} with cut-off parameter~$r = 0$.

\begin{example}[The midpoint method] \label{exa:midpoint}
	Let us choose $s = 1$.
		The method is of second order if and only if we choose $b_1 = 1$ and $a_{11} = 1/2$.
	This method is also symmetric, since the substitutions~$h \to -h$, $q_0 \leftrightarrow q_1$, $\mu_0 \leftrightarrow \mu_1$, $Y \to -Y$ in \Fref{eq:JCP-int} yields the same method after some manipulation of the equations.
	This property is utilized in \Fref{sec:order-tests} to achieve high order by composition.
\end{example}

\section{Higher order integrators}
The methods discussed in the previous section were limited to at most second order.
To obtain higher order integrators, we consider two approaches, based on two well-known classes of Lie group integrators.

The first approach is based on the Runge--Kutta--Munthe-Kaas (RKMK) methods.
This approach was already considered by Bou-Rabee and Marsden in 2009 \cite{bourabee09}.
The work in the present article builds on the work by Bou-Rabee and Marsden, and examines in detail the case when the cut-off parameter $r$ ($q$ in \cite{munthekaas99, bourabee09}) in the RKMK method is larger than $0$, and provides a complete order theory for variational methods based on RKMK methods.

The second approach is based on Crouch--Grossman (CG) methods. 
This approach has to our knowledge not been explored before.
We show that these methods can achieve arbitrarily high order, but the complete order theory of these methods remains unresolved.

\subsection{Variational Runge--Kutta--Munthe-Kaas integrators}
\label{sec:varRKMK}

A popular class of Lie group integrators is the class of RKMK integrators. 
For our purposes, these integrators can be written
\begin{equation}
 \begin{aligned}
    x_i &= h\sum_{j=1}^s a_{ij} \dexp^{-1}_{(r),x_j} \xi_j, \\
    Q_i &= \exp(x_i)q_0, \\
    \xi_i &= f(Q_i), \qquad \qquad \qquad i=1,\dotsc, s,\\
    Y &= h\sum_{i=1}^s b_i \dexp^{-1}_{(r),x_i} \xi_i, \\
    q_1 &= \exp(Y)q_0,
 \end{aligned}
\label{eq:RKMKint}
\end{equation}
where 
\[
\dexp^{-1}_{(r), x}=  \id- \frac 12 \ad_{x} + \sum_{k=2}^r \frac{B_k}{k!} (\ad_{x})^k
\]
is the Taylor series approximation to $\dexp^{-1}_{x}$, and $a_{ij}, b_i$ are the coefficients of a Runge--Kutta method.
If the RK method is of order $p$ and $r\ge p-2$, the resulting Lie group integrator is of order $p$ as well \cite[Theorem~IV.8.4]{hairer06}.

Variational methods based on RKMK methods were considered by Bou-Rabee and Marsden in \cite{bourabee09}, though the methods they present in detail are at most second order, since they only consider the case $r=0$.
The methods in this case are essentially the methods considered in \Fref{sec:secondorder}. 
For $r>0$, some complications arise for the variational integrator, since the $x_i$ are not explicitly given by $\xi_1,\dotsc, \xi_s$.
Our solution is to treat both $x_i$ and $\xi_i$ as unknowns and the equations for $x_i$ in \Fref{eq:RKMKint} as restrictions.
The Lagrange multipliers $\lambda_i$ corresponding to the equations for $x_i$ cannot be eliminated from the equations in a general manner, so the dimension of the non-linear equation to be solved at each step is larger than that of the corresponding symplectic method applied to $T^\ast \RR^n$, or the simpler integrator with $r=0$.

In the variational integrator we set the discrete Lagrangian to
\[L_h= h\sum_{j=1}^s b_j \ell(Q_j, \xi_j),\]
and let the constraints be given by \Fref{eq:varconstraints}, that is
\begin{align*}
    q_1 &= \exp(Y)q_0, \\
    Q_i &= \exp(X_i)q_0, \\
\shortintertext{and}
Y &= h\sum_{j=1}^s b_j \dexp^{-1}_{(r),x_j} \xi_j, \\
X_i &= h\sum_{j=1}^s a_{ij} \dexp^{-1}_{(r),x_j} \xi_j, \quad \quad \quad i=1,\dotsc, s,
\end{align*}
where $x_i = \log\bigl(Q_i q_0^{-1}\bigr)$. 
Note that on the solution set of the constraints, $X_i = x_i$.

Applying the variational equations from \Fref{eq:varprinciple}, the integrator is given by
\begin{equation}
  \begin{aligned}
  \mu_0 &= \Bigl(\bigl(\dexp^{-1}_{-Y}\bigr)^\ast - h\sum_i b_i \bigl(\dexp_{-X_i}^{-1}\bigr)^\ast \circ P^\ast_{(r)}(X_i, \xi_i) \Bigr)\Lambda \\
         &\phantom{=}+\sum_j \Bigl(\bigl(\dexp_{-X_j}^{-1}\bigr)^\ast - h \sum_i a_{ji}  \bigl(\dexp_{-X_i}^{-1}\bigr)^\ast \circ P^\ast_{(r)}(X_i, \xi_i) \Bigr) \lambda_j, \\
  hb_i n_i &= -\bigl(\dexp^{-1}_{X_i}\bigr)^\ast \lambda_i + h b_i \bigl(\dexp_{X_i}^{-1}\bigr)^\ast \circ P^\ast_{(r)}(X_i, \xi_i)\Lambda \\
         &\phantom{=}+ h \sum_j a_{ji}  \bigl(\dexp_{X_i}^{-1}\bigr)^\ast \circ P^\ast_{(r)}(X_i, \xi_i)\lambda_j, \\
  hb_i M_i &= h\bigl(\dexp^{-1}_{(r),X_i}\bigr)^\ast\Bigl(b_i\Lambda + \sum_j a_{ji} \lambda_j\Bigr), \\
  Q_i &= \exp(X_i)q_0, \\
  (\xi_i, n_i) &= f(Q_i, M_i), \qquad \qquad \qquad i=1,\dotsc, s,\\
  q_1 &= \exp(Y)q_0, \\
  \mu_1 &= \bigl(\dexp^{-1}_{Y}\bigr)^\ast \Lambda,
  \end{aligned}
\label{eq:VRKMK1}
\end{equation}
where $P_{(r)}^{\ast}(x, \xi)$ is a polynomial in $\ad_{x}^\ast$ and $\ad_{\xi}^\ast$ of degree $r$, defined as the adjoint of the partial derivative of $\dexp^{-1}_{(r), x} \xi$ with respect to $x$. 
Specifically,
\[
 \begin{aligned}
    P_{(0)}^\ast(x,\xi) &= 0, \\
    P_{(1)}^\ast(x,\xi) &= \frac 12 \ad^\ast_{\xi}, \\
    P_{(2)}^\ast(x,\xi) &= \frac 12 \ad^\ast_{\xi}-\frac 16 \ad^\ast_\xi \ad^\ast_x + \frac 1{12} \ad^\ast_x \ad^\ast_\xi, \\
    P_{(r)}^\ast(x,\xi) &=  \frac 12 \ad^\ast_{\xi}- \sum_{k=2}^{r} \frac{B_k}{k!} \sum_{i=0}^{k-1} \ad^\ast_{\ad^i_x \xi} \bigl( \ad^\ast_x \bigr)^{k-i-1}.
\end{aligned}
\]
By applying $\Ad^\ast_{\exp(X_i)}$ to both sides of the second equation in \Fref{eq:VRKMK1}, we see that the first equation can be simplified. 
Using this and rearranging the rest of the equations while assuming that $b_i \neq 0$, we arrive at the set of equations
\begin{equation}
\begin{aligned}
  \Lambda &= \dexp^\ast_{-Y}\Bigl(\mu_0 + h  \sum_i b_i\Ad^\ast_{\exp(X_i)}n_i\Bigr),\\
  \lambda_i &= -hb_i \dexp_{X_i}^\ast n_i + hP^\ast_{(r)}(X_i, \xi_i)\Bigl(b_i \Lambda + \sum_j a_{ji} \lambda_j\Bigr),\\
  M_i &= \frac{1}{b_i}\bigl(\dexp^{-1}_{(r),X_i}\bigr)^\ast\Bigl(b_i\Lambda + \sum_j a_{ji} \lambda_j\Bigr), \\
  X_i &= h \sum_j a_{ij} \dexp^{-1}_{(r), X_j} \xi_j,\\
  (\xi_i, n_i) &= f\bigl(\exp(X_i) q_0, M_i\bigr), \qquad \qquad \qquad i=1,\dotsc, s,\\
  Y &= h \sum_i b_i \dexp^{-1}_{(r), X_i} \xi_i, \\
  q_1 &= \exp(Y)q_0, \\
  \mu_1 &= \Ad^{*}_{\exp(-Y)}\bigl(\mu_0 + h  \sum_i b_i\Ad^\ast_{\exp(X_i)}n_i\bigr),
\end{aligned}
\label{eq:VRKMKgenq}
\end{equation}
which define a symplectic integrator on $\coT G$.
The identity $\dexp_{x}\circ \dexp^{-1}_{-x} = \Ad_{\exp(x)}$ was used to obtain the last line of the equations above.
We call the integrators defined by \Fref{eq:VRKMKgenq} \emph{variational Runge--Kutta--Munthe-Kaas methods} or VRKMK methods for short.
One can easily check that if the Lie group is abelian, a VRKMK method simplifies to a symplectic partitioned Runge--Kutta method.

In implementations of these methods, it is required that $\exp\from\g \to G$ and $\dexp^{*} \from \g\times \g^{*} \to \g^{*}$ are calculated to machine precision to obtain symplecticity.
In the numerical tests of \Fref{sec:num}, we choose $G = \LieSO(3)$ and use Rodrigues' formula \cite[Section~9.2]{marsden99-1} to calculate these expressions.
In a general setting, calculating these expressions usually involves analytic functions of matrices.
The equations defining the integrator can be solved as a set of non-linear equations in the unknowns $X_i, M_i$ and $\lambda_i$, $i=1,\dotsc, s$, as the other quantities in \Fref{eq:VRKMKgenq} are given explicitly in terms of the aforementioned variables as well as $q_0$ and $\mu_0$. 
If $G$ is an $n$-dimensional Lie group, this is a total of $3ns$~scalar unknowns.
Other choices of independent and dependent unknowns are possible. One could reduce the number of unknowns  by starting with coefficients of an explicit RK method and thereby obtain explicit expressions for $x_i$ and be able to eliminate the $\lambda_i$ in the integrator.
The variational method based on an explicit RK method would still be implicit, however, and due to the order conditions presented later in this article, the increased number of stages required for a particular order would offset the reduction in number of unknowns per stage obtained by using an explicit method as the underlying method, so it is unclear if this simplification is useful in practice.

\subsubsection{Alternative approach} \label{sec:pade}
The authors have discovered an alternative approach which reduces the number of unknowns in the equations \Fref{eq:VRKMKgenq}.
The alternative approach requires a modification of the variational principle described in \Fref{sec:varprinciple}, and full details and analysis goes beyond the scope of this article.

A modification of the RKMK methods can be obtained by replacing the Taylor approximation~$\dexp^{-1}_{(r)}$ in \Fref{eq:RKMKint} with a Pad\'{e} approximation.
After trivial manipulations, the modified RKMK method is
\begin{equation*}
 \begin{aligned}
    x_i &= h\sum_j a_{ij} \tilde{\xi}_j, \\
    Q_i &= \exp(x_i)q_0, \\
    \dexp_{(r),x_i} \tilde{\xi}_i &=  f(Q_i), \qquad \qquad \qquad i=1,\dotsc, s,\\
    Y &= h\sum_i b_i \tilde{\xi}_i, \\
    q_1 &= \exp(Y)q_0,
 \end{aligned}
\end{equation*}
where $\dexp_{(r),\xi} = 1+ \frac 12 \ad_\xi +\dotsb \frac{1}{(r+1)!} \ad_\xi^r$.
In this formulation, the $x_j$ are explicit in the $\tilde{\xi}_j$, so there is no need to introduce restrictions for the equations  $x_i = h\sum_j a_{ij} \tilde{\xi}_j$.

The discrete Lagrangian in this formulation is
\[L_h(q_0,q_1)=h\sum_j b_j \ell(Q_j, \dexp_{(r),x_j} \tilde{\xi}_j),
\]
which is not of the format \Fref{eq:RKaction} discussed in \Fref{sec:varprinciple}.
Therefore the general formulae for integrators \Fref{eq:varprinciple} derived earlier do not apply.
However, the general idea can still be pursued, and the resulting integrator can be formulated on the Hamiltonian side.



\subsection{Variational Crouch--Grossman integrators}

The methods of Crouch and Grossman form another important class of Lie group methods.
Crouch and Grossman formulated their integrators in terms of rigid frames, i.e. finite collections of vector fields on a manifold.
On a Lie group, a suitable rigid frame is a basis for the right-invariant vector fields on $G$ corresponding to a basis of $\g$.
In this setting, the Crouch--Grossman methods can be defined as follows.
Let $b_i$, $a_{ij}$ be the coefficients of an $s$-stage RK method.
The Crouch--Grossman method~\cite[Section~IV.8.1]{hairer06} with the same coefficients is defined by the equations
\begin{align*}
 Q_i &= \exp(ha_{is} \xi_s) \dotsm \exp(ha_{i1} \xi_1) q_0, \\
 \xi_i &= f(Q_i), \\
q_1 &= \exp(hb_s \xi_s) \dotsm \exp(hb_1 \xi_1)q_0.
\end{align*}
The order of a CG method is determined by the order conditions developed in \cite{owren99}.

Using the general format~\Fref{eq:varprinciple}, we set the discrete Lagrangian to
\begin{equation}
	L_h(q_0, q_1) = h \sum_{i = 1}^s b_i \ell(Q_i, \xi_i), \label{eq:cg-discrete-lagrangian}
\end{equation}
with $b_i \neq 0$ and constraints given by \Fref{eq:varconstraints}, that is
\begin{align*}
    q_1 &= \exp(Y)q_0, \\
    Q_i &= \exp(X_i)q_0, \\
    \shortintertext{and}
    Y   &= \log\bigl(\exp(h b_s \xi_s) \dotsm \exp(h b_1 \xi_1)\bigr), \\
    X_i &= \log\bigl(\exp(h a_{is} \xi_s) \dotsm \exp(h a_{i1} \xi_1)\bigr).
\end{align*}
Inserting this into the equations defining a variational integrator \Fref{eq:varprinciple}, we obtain
\begin{align*}
	\mu_0 &= \bigl(\dexp_{-Y}^{-1}\bigr)^{*} \Lambda + \sum_j \bigl(\dexp_{-X_j}^{-1}\bigr)^{*} \lambda_j, \\
	h b_i n_i &= -\bigl(\dexp_{X_i}^{-1}\bigr)^{*} \lambda_i, \\
	h b_i M_i &= h b_i \dexp_{h b_i \xi_i}^{*} \circ \Ad^{*}_{\exp(h b_s \xi_s) \dotsm \exp(h b_{i + 1} \xi_{i + 1})} \circ \bigl(\dexp_Y^{-1}\bigr)^{*} \Lambda \\
	&\phantom{=} + h \sum_j a_{ji} \dexp_{h a_{ji} \xi_i}^{*} \circ \Ad^{*}_{\exp(h a_{js} \xi_s) \dotsm \exp(h a_{j,i + 1} \xi_{i + 1})} \circ \bigl(\dexp_{X_j}^{-1}\bigr)^{*} \lambda_j, \\
	(\xi_i, n_i) &= f(Q_i, M_i), \\
	Q_i &= \exp(X_i) q_0, \\
	q_1 &= \exp(Y) q_0, \\
	\mu_1 &= \bigl(\dexp_Y^{-1}\bigr)^{*} \Lambda.
\end{align*}
Eliminating $\Lambda$ and $\lambda_j$, and rearranging, we get the integrator
\begin{equation}
\begin{gathered}
	q_1 = q^s, \qquad q^j = \exp(h b_j \xi_j) q^{j - 1}, \qquad q^0 = q_0, \\
	Q_i = Q_{is}, \qquad Q_{ij} = \exp(h a_{ij} \xi_j) Q_{i, j - 1}, \qquad Q_{i0} = q_0, \\
	(\xi_i, n_i) = f(Q_i, M_i), \\
	\bar \mu_0 = \Ad^{*}_{q_0} \mu_0, \qquad \bar \mu_1 = \Ad^{*}_{q_1} \mu_1, \qquad \bar n_i = \Ad^{*}_{Q_i} n_i, \\
	\bar \mu_1 = \bar \mu_0 + h \sum_{j = 1}^s b_j \bar n_j, \\
	M_i = \dexp^{*}_{h b_i \xi_i} \circ \Ad^{*}_{(q^i)^{-1}} \bar \mu_1 - h \sum_{j = 1}^s \frac{b_j a_{ji}}{b_i} \dexp^{*}_{h a_{ji} \xi_i} \circ \Ad^{*}_{Q_{ji}^{-1}} \bar n_j,
\end{gathered}
\label{eq:vcg-integrator}
\end{equation}
for all $i = 1, \dotsc, s$.
We call the integrators defined by \Fref{eq:vcg-integrator} \emph{variational Crouch--Grossman methods} or VCG methods for short.
The last equation in \Fref{eq:vcg-integrator} can also be written as
\begin{equation} \label{eq:M_j-with-mu_0}
	M_i = \dexp^{*}_{h b_i \xi_i} \circ \Ad^{*}_{(q^i)^{-1}} \bar \mu_0 + h \sum_{j = 1}^s b_j \left(\dexp^{*}_{h b_i \xi_i} \circ \Ad^{*}_{(q^i)^{-1}} - \frac{a_{ji}}{b_i} \dexp^{*}_{h a_{ji} \xi_i} \circ \Ad^{*}_{Q_{ji}^{-1}}\right) \bar n_j,
\end{equation}
which we will need in the order analysis of \Fref{prn:vcg-order}.

In the case that the Lie group is abelian, the integrator simplifies to the same symplectic, partitioned RK method as in the abelian case for the VRKMK integrator.

%

\section{Order analysis}
In analyzing the order of variational methods we use variational error analysis as described by Marsden and West \cite[Section~2.3]{marsden01}.
We recite two definitions from this reference which are useful in the following sections.
The \emph{exact discrete Lagrangian} is given by
\[L_h^{\mathrm{E}}(q_0, q_1)= \int_{0}^h L\bigl(q(t),  \dot{q}(t)\bigr)\, \dd t,\]
where $q(t)$ is the solution to the Euler--Lagrange equations with $q(0)=q_0$, $q(h)=q_1$.
A discrete Lagrangian $L_h$ is said to be \emph{of order $p$} if
\[L_h\bigl(q(0), q(h)\bigr) = L_h^{\mathrm{E}}\bigl(q(0), q(h)\bigr) + \mathcal{O}(h^{p+1}),\]
for all solutions $q(t)$ of the Euler--Lagrange equations.
The following theorem is a special case of \cite[Theorem~2.3.1]{marsden01}.%
\footnote{Patrick and Cuell \cite{cuell09} demonstrate an inaccuracy in the proof in \cite{marsden01}. However, they also show that the relevant result still holds.}
\begin{theorem}
\label{thm:varorder}
Given a regular Lagrangian $L$ and a discrete Lagrangian $L_h$ of order $p$, then the symplectic integrator defined by $L_h$ is of order~$p$.
\end{theorem}

Both classes of methods presented in this article depend on Butcher coefficients $a_{ij}$ and $b_i$.
Furthermore, when applied to an abelian Lie group (for instance $\RR^n$), both classes become symplectic, partitioned RK methods where the position is integrated with the RK method with coefficients $a_{ij}$ and $b_i$, while the momentum is integrated with the RK method with coefficients~$\hat{a}_{ij} = b_j - b_j a_{ji} / b_i$ and $\hat{b}_i = b_i$.
The order conditions for SPRK methods have been explored in detail by Murua \cite{murua97}.
Since an abelian Lie group is a special case, the order of the SPRK method is an upper bound for the order of the variational Lie group method with the same coefficients.
The order of the underlying Lie group method is also an upper bound according to \Fref{prn:orderlimit}.
\subsection{Order of VRKMK integrators}
\label{sec:orderRKMK}

The VRKMK methods described in \Fref{sec:varRKMK} are fully described by the Butcher coefficients~$a_{ij}$ and $b_i$, and the cut-off parameter $r$.
The cut-off parameter $r$ limits the order of the RKMK method \Fref{eq:RKMKint} on $G$. 
The order of the RKMK method is the minimum of the order of the RK method based on the same coefficients and $r+2$. \cite[Section IV.8.2]{hairer06}

As explained above, the order of the VRKMK method is bounded from above by the order of the SPRK method based on the same Butcher coefficients, and by the order of the RKMK method.
Since the order conditions for RK methods for a particular order form a subset of the order conditions for the SPRK method, we can a priori say that the order of the VRKMK method is bounded from above by the order of the SPRK method and $r+2$.
\Fref{thm:orderRKMK} states that in the case of a regular Lagrangian, the order of the VRKMK method is in fact the minimum of these two bounds.
The proof of this theorem relies on the following lemma.
\begin{lemma}
 Assume that the continuous Lagrangian is regular, and that the SPRK method based on the coefficients $a_{ij}$ and $b_i$ is of order $d$.
 Then the discrete Lagrangian of the SPRK method \Fref{eq:RKaction} is also of order $d$.
 \label{lem:VRKorder}
\end{lemma}
\begin{proof}
 Let $H(q,p)$ be a regular Hamiltonian, $L(q,\dot{q})=\langle p,\dot{q} \rangle -H(q,p)$ the corresponding Lagrangian, and $\bigl(q(t), p(t)\bigr)$ an exact solution to the Hamiltonian system.
 The resulting integrator is order $d$ accurate if and only if the original coefficients $b_i$ and $a_{ij}$ together with $\hat{b}_i=b_i$ and $\hat{a}_{ij}=b_j-b_j a_{ji}/b_i$ fulfil the order conditions up to order $d$ for a partitioned Runge--Kutta method.
 We will apply the partitioned Runge--Kutta method to the system
 \begin{align*}
 \begin{bmatrix} \dot{q} \\ \dot{S} \end{bmatrix}  &=\begin{bmatrix} \pdiff{H}{p}\\ L\bigl(q,\dot q\bigr) \end{bmatrix},\\
  \dot{p}&=-\pdiff{H}{q},
 \end{align*}
 where the $(q,S)$-component is integrated using the coefficients $b_i$, $a_{ij}$, and the $p$-component is integrated using the coefficients $\hat{b}_i$, $\hat{a}_{ij}$.\footnote{Since $\hat{b}_i=b_i$ and the right hand side is independent of $S$, we could instead have grouped $S$ with $p$ without any change.}
 These are simply the Hamiltonian equations augmented with the differential equation for the action integral $S$.
 
 As starting values, we use $q_0 = q(0), p_0 = p(0)$ and $S_0 = S(0) = 0$.
 The exact solution of the system at $t=h$ is given by the solution to the Hamiltonian equation, $q(h)$, $p(h)$, and 
 \[S(h)=\int_0^h L(q, \dot q) \, \dd t = L_h^\mathrm{E} \bigl(q_0, q(h)\bigr).\]
 The numerical solution obtained with one step of the partitioned Runge--Kutta method is, using the notation of \Fref{sec:gen-format},
 \begin{gather*}
  q_1= q(h) + \mathcal{O}(h^{d+1}), \qquad p_1 = p(h) + \mathcal{O}(h^{d+1}) \\
  S_1 = h\sum_{i=1}^s b_i L(Q_i, \dot{Q}_i) = L_h(q_0, q_1) =  L_h^\mathrm{E}\bigl(q_0, q(h)\bigr)+\mathcal{O}(h^{d+1}),
 \end{gather*}
 since the method is order $d$.
 Using Taylor series expansion and that $p_1= \DD_2L_h(q_0, q_1)$, we see that $L_h(q_0, q_1) - L_h\bigl(q_0, q(h)\bigr)=\langle p_1, q_1-q(h) \rangle + \mathcal{O}\bigl(q_1-q(h)\bigr)^2 = \mathcal{O}(h^{d+1})$, which completes the proof.
\end{proof}

\begin{theorem}
 If the symplectic, partitioned Runge--Kutta method based on the coefficients $a_{ij}$ and $b_i$ is of order at least $p$, and the cut-off parameter $r$ satisfies $r\geq p-2$, then the variational Runge--Kutta--Munthe-Kaas method with the same coefficients is at least of order $p$ for regular Hamiltonians.
\label{thm:orderRKMK}
\end{theorem}
\begin{proof}
The proof consists of two steps. 
We introduce the limit case where the cut-off parameter~$r$ goes to infinity, that is, the method where $\dexp^{-1}_{(r), x}$ is replaced by $\dexp^{-1}_x$.
To distinguish between the two RKMK methods, we will denote the ``full'' RKMK method by RKMK$(\infty)$, and the RKMK method with cut-off parameter~$r$ by RKMK$(r)$.
The variational integrators based on the two methods are denoted VRKMK$(\infty)$ and VRKMK$(r)$, respectively.

In the first step, we show that the discrete Lagrangian which defines the VRKMK$(\infty)$ method is of order $p$.
The proof relies on two facts.
Firstly, that the discrete Lagrangian of the SPRK method is of order $p$.
Secondly, that the discrete Lagrangians of the VRKMK$(\infty)$ and of a special case of the SPRK method are obtained as extremal values of the same object function and under the same constraints.

In the second part, we show that if we apply the VRKMK$(\infty)$ and VRKMK$(r)$ methods to the same initial values, their difference after one step goes to zero as $\mathcal{O}(h^{r+3})$.

Let $q\from [0, a] \rightarrow G$ be a solution to the Euler--Lagrange equation with $q(0) = q_0$, and assume that $a>0$ is sufficiently small such that $\sigma(t) =\log\bigl(q(t)q_0^{-1}\bigr)$ is uniquely defined for all $t\in [0, a]$.
The exact discrete Lagrangian is given by
\begin{equation}
L_h^{\text{E}}\bigl(q_0, q(h)\bigr) = \int_0^h \ell\bigl(q(t), \xi(t)\bigr) \, \dd t,
\label{eq:LEgroup}
\end{equation}
where $\xi(t) = \dot{q}(t)\cdot q(t)^{-1}$.
If we define $\tilde{\ell}\from T\mathfrak{g} \to \RR$ as
\begin{equation}
\tilde{\ell}(\sigma, \dot{\sigma}) = \ell\bigl(\exp(\sigma)q_0, \dexp_\sigma \dot{\sigma}\bigr),
\label{eq:ltildedef}
\end{equation}
we can rewrite \Fref{eq:LEgroup} as $L_h^{\text{E}}\bigl(q_0, q(h)\bigr) = \tilde{L}_h^{\text{E}}\bigl(0, \sigma(h)\bigr)$ where
\begin{equation*}
\tilde{L}_h^{\text{E}}\bigl(0, \sigma(h)\bigr) = \int_0^h \tilde{\ell}\bigl(\sigma(t), \dot{\sigma}(t)\bigr) \, \dd t.
\label{eq:LEalg}
\end{equation*}
This is an exact discrete Lagrangian on the vector space $\mathfrak{g}$, which we approximate by the action sum arising from the underlying RK method,
\begin{equation}
 \tilde{L}_h^{\text{RK}}\bigl(0, \sigma(h)\bigr) = h\sum_i b_i \tilde{\ell}(y_i, \eta_i),
 \label{eq:LRKalg}
\end{equation}
where $y_i = h\sum_j a_{ij} \eta_j$, $i=1,\dotsc, s$ and the sum is extremized under the constraint $\sigma(h) = h \sum_i b_i \eta_i$.
Under the assumptions of the theorem, the order of the SPRK method is at least $p$, so by \Fref{lem:VRKorder}, the discrete Lagrangian of the SPRK method is order $p$ accurate,
\[\tilde{L}_h^\text{RK}\bigl(0, \sigma(h)\bigr) = \tilde{L}_h^\text{E}\bigl(0, \sigma(h)\bigr) + \mathcal{O}(h^{p+1}).\]
Inserting \Fref{eq:ltildedef} into \Fref{eq:LRKalg} gives 
\begin{equation}
 L_h^{\text{RK}}\bigl(q_0, q(h)\bigr) = \tilde L_h^\text{RK}\bigl( 0, \sigma(h) \bigr) = h\sum_i b_i \ell\bigl(\exp(y_i) q_0, \dexp_{y_i} \eta_i\bigr),
\label{eq:LRKgroup}
\end{equation}
where the sum is extremized under the constraint $q(h) = \exp\bigl(h \sum_i b_i \eta_i\bigr) q_0$.

Now, the discrete action sum arising from RKMK$(\infty)$ is
\begin{equation}
 L_h^{\text{RKMK}(\infty)}\bigl(q_0, q(h)\bigr) = h\sum_i b_i \ell\bigl(\exp(X_i) q_0, \xi_i\bigr),
\label{eq:LRKMKinfgroup}
\end{equation}
which is extremized under the constraints
\[
 \begin{aligned}
  X_i &= h\sum_j a_{ij} \dexp^{-1}_{X_j} \xi_j, \quad \quad i=1,\dotsc, s, \\
  q(h) &= \exp\Bigl(h\sum_i b_i \dexp^{-1}_{X_i} \xi_i\Bigr)q_0.
 \end{aligned}
\]
We see that under the identifications
\[y_i = X_i, \qquad \eta_i = \dexp^{-1}_{X_i} \xi_i,\]
the objective functions \Fref{eq:LRKgroup} and \Fref{eq:LRKMKinfgroup} are identical and are extremized under the same constraints. 
Thus their extremal values are identical and we have proved
\[ 
 \begin{aligned}
    L_h^{\text{RKMK}(\infty)}\bigl(q_0, q(h)\bigr) &= L_h^{\text{RK}}\bigl(q_0, q(h)\bigr)=\tilde{L}_h^{\text{RK}}\bigl(0, \sigma(h)\bigr) \\
					 &=\tilde{L}_h^\text{E}\bigl(0, \sigma(h)\bigr) + \mathcal{O}(h^{p+1})= L_h^{\text{E}}\bigl(q_0, q(h)\bigr) + \mathcal{O}(h^{p+1}),
\end{aligned}
\]
concluding the first part of the proof.

For the second part of the proof, we consider the integrator in \Fref{eq:VRKMKgenq} and the variational integrator based on RKMK$(\infty)$ with the same initial data $(q_0, \mu_0)$.
Let $\xi_i, n_i, X_i, M_i, \lambda_i, \Lambda$ and $Y$ be as in \Fref{eq:VRKMKgenq}, and $\xi_i^{(\infty)}$, etc.\ be the corresponding quantities in VRKMK$(\infty)$.

We define $\delta \xi_i = \xi_i-\xi_i^{(\infty)}$ and so on, and consider the difference between VRKMK$(\infty)$ and VRKMK$(r)$. 
Since $q_1 = \exp(Y)q_0$ and $\mu_1 = \bigl(\dexp_Y^{-1}\bigr)^{*} \Lambda$, the leading order of the difference between the two integrators is given by the leading orders of $\delta \! Y$ and $\delta \! \Lambda$.
It is clear from the equations in \Fref{eq:VRKMKgenq} defining the integrator that as $h\rightarrow 0$, $\lambda_i, \lambda_i^{(\infty)}, X_i, X_i^{(\infty)}, Y$ and $Y^{(\infty)}$ all go to zero as $\OO(h)$.
Furthermore, $\delta \xi_i , \delta \! n_i , \delta \! M_i$ and $\delta \! \Lambda$ must also go to zero as $h\rightarrow 0$.
By using the expressions in \Fref{eq:VRKMKgenq}, the equations~$\xi_i = \xi_0 + \OO(h)$, $n_i = n_0 + \OO(h)$, $\Lambda = \mu_0 + \OO(h)$, and $X_i = h c_i \xi_0 + \OO(h^2)$, and the series expansions of $\dexp^{*}$ and $\Ad^{*}_{\exp(\cdot)}$, we find that
\begin{align*}
	\delta \! \Lambda                &= -\tfrac{1}{2} \ad^{*}_{\delta \! Y} \mu_0 + h \sum_i b_i \bigl(\ad^{*}_{\delta \! X_i} n_0 + \delta \! n_i\bigr) + \text{higher order terms}, \\
	\delta \! M_i                    &= -\tfrac{1}{2} \ad^{*}_{\delta \! X_i} \mu_0 - \frac{B_{r + 1}}{(r + 1)!} h^{r + 1} c_i^{r + 1} \bigl(\ad^{*}_{\xi_0}\bigr)^{r + 1} \mu_0 + \delta \! \Lambda + \sum_j \frac{a_{j i}}{b_i} \delta \! \lambda_j + \OO(h^{r + 2}) + \hot, \\
	\delta \! \lambda_i              &= -h b_i \bigl(\tfrac{1}{2} \ad^{*}_{\delta \! X_i} n_0 + \delta \! n_i\bigr) + h b_i \frac{B_{r + 1}}{(r + 1)!} h^r c_i^r \bigl(\ad^{*}_{\xi_0}\bigr)^{r + 1} \mu_0 \\
	                                 &\quad + h b_i \Bigl(\bigl(\tfrac{1}{2} \ad^{*}_{\delta \xi_i} - \tfrac{1}{6} \ad^{*}_{\xi_0} \ad^{*}_{\delta \! X_i} + \tfrac{1}{12} \ad^{*}_{\delta \! X_i} \ad^{*}_{\xi_0}\bigr) \mu_0 + \tfrac{1}{2} \ad^{*}_{\xi_0} \delta \! \Lambda\Bigr) + \OO(h^{r + 2}) + \hot, \\
	\delta \! X_i                    &= h \sum_j a_{i j} \delta \xi_j + \OO(h^{r + 3}) + \hot, \\
	(\delta \xi_i, \delta \! n_i) &= T_{(q_0,\mu_0)} f(\delta \! X_i \cdot q_0, \delta \! M_i) + \hot, \\
	\delta \! Y                      &= h \sum_i b_i \bigl(\delta \xi_i - \tfrac{1}{2} \ad_{\delta \! X_i} \xi_0 \bigr) + \OO(h^{r + 3}) + \hot
\end{align*}
In the equations above, ``higher order terms (h.o.t.)'' denote terms that are dominated by at least one of the preceding terms.

We continue by combining the equations and dropping terms of higher order.
Consider the equation for $\delta \xi_i$, and insert the expression for $\delta \! X_i$.
We obtain
\begin{align*}
	\delta \xi_i &= \pdiff{\! f_1}{q} (\delta \! X_i \cdot q_0) + \pdiff{\! f_1}{\mu} (\delta \! M_i) + \hot \\
	                &= \pdiff{\! f_1}{q} \Bigl(h \sum_j a_{ij} \delta \xi_j \cdot q_0\Bigr) + \pdiff{\! f_1}{\mu} (\delta \! M_i) + \OO(h^{r + 3}) + \hot \\
	                &=  \pdiff{\! f_1}{\mu} (\delta \! M_i) + \OO(h^{r + 3}) + \hot,
\shortintertext{and}
	\delta \! X_i   &= h \sum_j a_{ij} \pdiff{\! f_1}{\mu} (\delta \! M_i) + \OO(h^{r + 3}) + \hot
\end{align*}
Similarly, successively we get
\begin{align*}
	\delta \! n_i       &= \pdiff{\! f_2}{q} (\delta \! X_i \cdot q_0) + \pdiff{\! f_2}{\mu} (\delta \! M_i) + \hot \\
	                    &= \pdiff{\! f_2}{\mu} (\delta \! M_i) + \OO(h^{r + 3}) + \hot, \\
	\delta \! Y         &= h \sum_i b_i \delta \xi_i + \OO(h^{r + 3}) + \hot \\
	                    &= h \pdiff{\! f_1}{\mu} \Bigl(\sum_i b_i \delta \! M_i\Bigr) + \OO(h^{r + 3}) + \hot,\\
	\delta \! \Lambda   &= -\tfrac{1}{2} \ad^{*}_{\delta \! Y} \mu_0 + h \sum_i b_i \delta \! n_i + \OO(h^{r + 3}) + \hot, \\
	\delta \! \lambda_i &= h b_i \biggl(\frac{B_{r + 1}}{(r + 1)!} h^r c_i^r \bigl(\ad^{*}_{\xi_0}\bigr)^{r + 1} \mu_0 + \tfrac{1}{2} \ad^{*}_{\delta \xi_i} \mu_0 - \delta \! n_i\biggr) + \OO(h^{r + 2}) + \hot, \\
	\delta \! M_i       &= \frac{B_{r + 1}}{(r + 1)!} h^{r + 1} \biggl(-c_i^{r + 1} + \sum_j \frac{a_{j i}}{b_i} b_j c_j^r\biggr) \bigl(\ad^{*}_{\xi_0}\bigr)^{r + 1} \mu_0  + \OO(h^{r + 2}).
\end{align*}
From the last equation, we see that
\begin{align*}
	\sum_i b_i \delta \! M_i &= \OO(h^{r + 2}),
\shortintertext{which yields, successively,}
	\delta \! Y              &= \OO(h^{r + 3}), \\
	\sum_i b_i \delta \! n_i &= \OO(h^{r + 2}), \\
	\delta \! \Lambda        &= \OO(h^{r + 3}),
\end{align*}
concluding the proof.
\end{proof}

An immediate consequence of the proof is that there exist methods in this class of arbitrarily high order.
For instance, the Gauss methods \cite[Section~II.1.3]{hairer06} form a class of Runge--Kutta methods which achieve arbitrarily high order.
Since these methods themselves are symplectic, $\hat{a}_{ij}= b_j - b_j a_{ji} / b_i = a_{ij}$, and the variational method based on a Gauss method is a partitioned Runge--Kutta method with the same coefficients for both position and momentum.
The variational method is equivalent to the Gauss method itself applied to the Hamiltonian ODE, and has therefore the same order as the Gauss method itself.
When $r$ is large enough, the VRKMK method based on the coefficients of a Gauss method achieves the same order as the Gauss method.

\subsection{Order of VCG integrators}

In this section, we will prove that there exist VCG integrators of any order.
To show this, we will need the following lemma.
\begin{lemma}[Composition of VCG integrators] 
	Let $(A^{(1)}, b^{(1)})$ and $(A^{(2)}, b^{(2)})$ be the Butcher tableaux of Runge--Kutta methods with $s^{(1)}$ and $s^{(2)}$ stages, and $\gamma$ a real number.
	The composition method formed by first applying the VCG method based on $(A^{(1)}, b^{(1)})$ with step length~$\gamma h$ and then the VCG method based on $(A^{(2)}, b^{(2)})$ with step length~$(1 - \gamma) h$, is a VCG method with Butcher tableau
	\bigskip
	\renewcommand{\arraystretch}{1.4}
	\[
	\begin{array}{c|cccccc}
		&                  &                &                          & \multicolumn{1}{|c}{} &                    &                             \\
		&                  & \gamma A^{(1)} &                          & \multicolumn{1}{|c}{} & 0                  &                             \\
		&                  &                &                          & \multicolumn{1}{|c}{} &                    &                             \\ \cline{2-7}
		& \gamma b_1^{(1)} & \cdots         & \gamma b_{s^{(1)}}^{(1)} & \multicolumn{1}{|c}{} &                    &                             \\
		& \vdots           &                & \vdots                   & \multicolumn{1}{|c}{} & (1-\gamma) A^{(2)} &                             \\
		& \gamma b_1^{(1)} & \cdots         & \gamma b_{s^{(1)}}^{(1)} & \multicolumn{1}{|c}{} &                    &                             \\
		\hline
		& \gamma b_1^{(1)} & \cdots         & \gamma b_{s^{(1)}}^{(1)} & (1-\gamma) b_1^{(2)}  & \cdots             & (1-\gamma)b_{s^{(2)}}^{(2)} \\
	\end{array}
	\]
	\bigskip
\end{lemma}
\begin{proof}
	Consider the discrete Lagrangians corresponding to the two VCG integrators that are to be composed,
	\begin{align*}
		L_h^{(1)}(q_0, q_1) &= h \sum_{i = 1}^{s^{(1)}} b_i^{(1)} \ell\bigl(Q_i^{(1)}, \xi_i^{(1)}\bigr), \\
	\shortintertext{with constraints}
		Q_i^{(1)} &= \exp\bigl(h a_{i{s^{(1)}}}^{(1)} \xi^{(1)}_{s^{(1)}}\bigr) \dotsm \exp\bigl(h a^{(1)}_{i1} \xi^{(1)}_1\bigr) q_0, \\
		q_1 &= \exp\bigl(h b^{(1)}_{s^{(1)}} \xi_{s^{(1)}}^{(1)}\bigr) \dotsm \exp\bigl(h b^{(1)}_1 \xi^{(1)}_1\bigr) q_0,
	\end{align*}
	and
	\begin{align*}
		L_h^{(2)}(q_0, q_1) &= h \sum_{i = 1}^{s^{(2)}} b_i^{(2)} \ell\bigl(Q_i^{(2)}, \xi_i^{(2)}\bigr), \\
	\shortintertext{with constraints}
		Q_i^{(2)} &= \exp\bigl(h a^{(2)}_{is^{(2)}} \xi^{(2)}_{s^{(2)}}\bigr) \dotsm \exp\bigl(h a^{(2)}_{i1} \xi^{(2)}_1\bigr) q_0, \\
		q_1 &= \exp\bigl(h b^{(2)}_{s^{(2)}} \xi^{(2)}_{s^{(2)}}\bigr) \dotsm \exp\bigl(h b^{(2)}_1 \xi^{(2)}_1\bigr) q_0,
	\end{align*}
	as well as the \emph{composition discrete Lagrangian}
	\[
		L_h^{(\mathrm{c})}(q_0, q_1) = L_{\gamma h}^{(1)} (q_0,\bar q) + L^{(2)}_{(1-\gamma) h} (\bar q,q_1),
	\]
	where $\bar q$ is chosen so that $L_h^{(\mathrm{c})}$ is extremized.
	It was proved in \cite[Theorem~2.5.1]{marsden01} that the integrator corresponding to $L_h^{(\mathrm{c})}$ is the composition method that results from composing the integrator corresponding to $L^{(1)}_{\gamma h}$ with the integrator corresponding to $L_{(1-\gamma)h}^{(2)}$.
	Denote by $(A^{(\mathrm{c})}, b^{(\mathrm{c})})$ the Butcher tableau with $s^{(\mathrm{c})} = s^{(1)} + s^{(2)}$ stages given above.
	Then
	\begin{align*}
		L_h^{(\mathrm{c})}(q_0, q_1) &= h \Biggl(\gamma \sum_{i = 1}^{s^{(1)}} b^{(1)}_i \ell\bigl(Q^{(1)}_i, \xi^{(1)}_i\bigr) + (1-\gamma) \sum_{i = 1}^{s^{(2)}} b^{(2)}_i \ell\bigl(Q^{(2)}_i, \xi^{(2)}_i\bigr) \Biggr) \\
		&= h \sum_{i = 1}^{s^{(\mathrm{c})}} b^{(\mathrm{c})}_i \ell\bigl(Q^{(\mathrm{c})}_i, \xi^{(\mathrm{c})}_i\bigr), \\
	\shortintertext{with constraints}
		Q_i^{(\mathrm{c})} &= \exp\bigl(h a^{(\mathrm{c})}_{i s^{(\mathrm{c})}} \xi^{(\mathrm{c})}_{s^{(\mathrm{c})}}\bigr) \dotsm \exp\bigl(h a^{(\mathrm{c})}_{i1} \xi^{(\mathrm{c})}_1\bigr) q_0, \\
		q_1 &= \exp\bigl(h b^{(\mathrm{c})}_{s^{(\mathrm{c})}} \xi^{(\mathrm{c})}_{s^{(\mathrm{c})}}\bigr) \dotsm \exp\bigl(h b^{(\mathrm{c})}_1 \xi^{(\mathrm{c})}_1\bigr) q_0.
	\end{align*}
\end{proof}

\begin{proposition}
	There exist methods of any order among the VCG integrators.
\end{proposition}
\begin{proof}
	From \cite[Section~II.4]{hairer06}, we know that if we compose a one-step method with itself using different step sizes, we can obtain arbitrarily high order, provided we choose the number of steps and the step sizes appropriately.
	Thus, we obtain VCG methods of any order by composition.
\end{proof}

\begin{proposition} \label{prn:vcg-order}
	For VCG integrators applied to regular Lagrangian problems, the order conditions for first and second order are the same as for the underlying Runge--Kutta method, i.e.\
	\begin{equation*}
		\sum_{i = 1}^s b_i = 1, \qquad \text{and} \qquad \sum_{i = 1}^s b_i c_i = \frac{1}{2},
	\end{equation*}
	where $c_i = \sum_j a_{ij}$.
\end{proposition}
\begin{proof}
	We use variational order analysis, as presented in \cite[Section~2.3]{marsden01}.
	Let the exact discrete Lagrangian be denoted
	\[
		L_h^\mathrm{E}\bigl(q_0, q(h)\bigr) = \int_0^h L\bigl(q(t), \dot q(t)\bigr) \, \dd t, \qquad \text{where} \quad q_0 = q(0).
	\]
	The exact discrete Lagrangian can be expanded in powers of $h$:
	\[
		L_h^\mathrm{E}\bigl(q_0, q(h)\bigr) = \sum_{k = 0}^\infty \frac{h^k}{k!} \biggl(\diff{^k}{h^k} L_h^\mathrm{E}\bigl(q_0, q(h)\bigr)\biggr\rvert_{h = 0}\biggr).
	\]
	From the right-trivialised HP equations~\Fref{eq:right-triv-HP} and \Fref{eq:f-function}, it is straight-forward to show that
	\[
		\diff{}{t} L\bigl(q(t), \dot q(t)\bigr) = \diff{}{t} \langle\mu, \xi\rangle = \langle \dot\mu, \xi\rangle + \langle\mu, \dot\xi\rangle = \langle f_2(z), f_1(z)\rangle + \biggl\langle\mu, \diff{}{t} f_1(z)\biggr\rangle,
	\]
	where $z = (q, \mu)$ and $f(z) = \bigl(f_1(z), f_2(z)\bigr)$.
	Thus, letting $(\xi_0, n_0) = f(z_0) = f(q_0, \mu_0)$, and using $\ell(q, \xi) = L(q, \dot q)$, we get
	\[
		L_h^\mathrm{E}\bigl(q_0, q(h)\bigr) = h \ell(q_0, \xi_0) + \frac{h^2}{2} \biggl(\langle n_0, \xi_0\rangle + \biggl\langle\mu_0, \pdiff{f_1}{q}(\xi_0 \cdot q_0) + \pdiff{f_1}{\mu}\bigl(n_0 - \ad^{*}_{\xi_0} \mu_0\bigr)\biggr\rangle\biggr) + \mathcal O(h^3).
	\]
	
	Similarly, we can expand the discrete Lagrangian in powers of $h$ by using \Fref{eq:cg-discrete-lagrangian} together with $Q_i\rvert_{h=0} = q_0$ and $\xi_i \rvert_{h=0} = \xi_0$:
	\begin{align*}
		L_h\bigl(q_0, q(h)\bigr) &= \sum_{k = 0}^\infty \frac{h^k}{k!} \biggl(\diff{^k}{h^k} L_h\bigl(q_0, q(h)\bigr)\biggr\rvert_{h = 0}\biggr) \\
		&= \sum_{k = 0}^\infty \frac{h^k}{k!} \biggl(\diff{^k}{h^k} h \sum_{i = 1}^s b_i \ell(Q_i, \xi_i) \biggr\rvert_{h = 0}\biggr) \\
		&= h \Bigl(\sum_i b_i\Bigr) \ell(q_0, \xi_0) + \frac{h^2}{2} \biggl(2 \sum_i b_i \diff{}{h} \ell(Q_i, \xi_i)\biggr\rvert_{h=0}\biggr) + \mathcal O(h^3).
	\end{align*}
	By comparing equal powers of the two expansions, we see that the first order condition is $\sum_i b_i = 1$, as in RK methods.
	The second term needs more work.
	We apply \Fref{eq:f-function} together with $n_i\rvert_{h=0} = n_0$ and $M_i\rvert_{h=0} = \mu_0$ and get
	\begin{align*}
		2 \sum_i b_i \diff{}{h} \ell(Q_i, \xi_i) \biggr\rvert_{h=0} &= 2 \sum_i b_i \biggl(\biggl\langle n_i \cdot Q_i, \diff{Q_i}{h}\biggr\rangle + \biggl\langle M_i, \diff{\xi_i}{h}\biggr\rangle\biggr)\biggr\rvert_{h=0} \\
		&= 2 \sum_i b_i \biggl\langle n_0 \cdot q_0, \diff{Q_i}{h} \biggr\rvert_{h=0}\biggr\rangle + 2 \biggl\langle \mu_0, \sum_i b_i \diff{\xi_i}{h} \biggr\rvert_{h=0} \biggr\rangle.
	\end{align*}
	We calculate the derivatives of $Q_i$ and $\xi_i$ with respect to $h$ using \Fref{eq:vcg-integrator}:
	\begin{gather*}
		\diff{Q_i}{h}\biggr\rvert_{h=0} = \sum_j a_{ij} \xi_0 \cdot q_0 = c_i \xi_0 \cdot q_0, \\
		\sum_i b_i \diff{\xi_i}{h}\biggr\rvert_{h=0} = \sum_i b_i \biggl( \pdiff{f_1}{q} \circ \diff{Q_i}{h}\biggr\rvert_{h=0} + \pdiff{f_1}{\mu} \circ \diff{M_i}{h}\biggr\rvert_{h=0} \biggr)
	\end{gather*}
	We also need the derivative of $M_i$. In this expression, we apply \Fref{eq:M_j-with-mu_0} and simplify using the first order condition:
	\[
		\sum_i b_i \diff{M_i}{h}\biggr\rvert_{h=0} = \Bigl(1 - \sum_i b_i c_i\Bigr) n_0 - \frac{1}{2} \ad^{*}_{\xi_0} \mu_0.
	\]
	Putting these equations together, we obtain
	\begin{multline*}
		2 \sum_i b_i \diff{}{h} \ell(Q_i, \xi_i) \biggr\rvert_{h=0} = 2 \sum_i b_i c_i \langle n_0, \xi_0 \rangle \\
		+ \biggl\langle \mu_0, 2 \sum_i b_i c_i \pdiff{f_1}{q}(\xi_0 \cdot q_0) + \pdiff{f_1}{\mu}\biggl(2 \Bigl(1-\sum_i b_i c_i\Bigr) n_0 - \ad^{*}_{\xi_0} \mu_0\biggr) \biggr\rangle.
	\end{multline*}
	Thus, to get second order, we need the second order RK condition $\sum_i b_i c_i = 1/2$.
\end{proof}

The computation for third order is similar, but much more complicated.
We give the third order conditions here, without proof.
\begin{proposition}
	For VCG methods applied to regular Lagrangian problems, using $b_i \hat a_{ij} + b_j a_{ji} = b_i b_j$ and $\hat c_i = \sum_{j = 1}^s \hat a_{ij}$, the conditions for third order are
	\begin{align*}
		\sum_{i = 1}^s b_i c_i^2 &= \frac{1}{3}, \\
		\sum_{i = 1}^s \sum_{j = 1}^s b_i a_{ij}c_j  &= \frac{1}{6}, \\
		\sum_{i = 1}^s b_i c_i \biggl( \sum_{j = 1}^{i - 1} b_j + \frac{b_i}{2} \biggr) &= \frac{1}{3}, \\
		\sum_{i = 1}^s b_i \hat c_i^2 &= \frac{1}{3}, \\
		\sum_{i = 1}^s b_i \hat c_i \biggl( \sum_{j = 1}^{i - 1} b_j + \frac{b_i}{2} \biggr) &= \frac{1}{3}, \\
		\sum_{i = 1}^s b_i^3 &= 0.
	\end{align*}
\end{proposition}
The first two conditions come from standard RK methods, the third condition comes from CG methods and the fourth comes from SPRK methods, while the final two conditions are new.
The last condition also appear in the order conditions for compositions of one-stage RK methods.
It is noteworthy that the final condition forces at least one of the weights $b_i$ to be negative.

The general order theory for VCG integrators is not complete and needs further study.

\section{Numerical tests}
\label{sec:num}

\begin{figure}
	\centering
	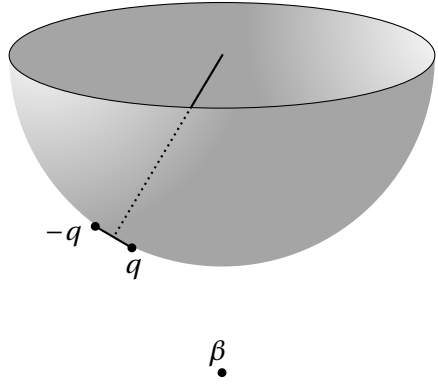
	\caption{Dipole on a stick}
	\label{fig:dipole-on-a-stick}
\end{figure}

To test our methods, we constructed a Hamiltonian test problem which we call ``dipole on a stick'' (see \Fref{fig:dipole-on-a-stick}).
The problem models a pendulum consisting of a long straight rod of length~$1$, with one end fixed (but freely rotating) at the origin, and a shorter rod of length $2\alpha$ with its centre attached perpendicularly to the long rod at the other end.
At each of the endpoints of the shorter rod there are charged particles with masses $m/2$ and electric charges~$\pm q$.
The rods are assumed to be massless.
The pendulum is affected by gravity in the negative~$\mathrm{e}_3$-direction and the electric field generated by a charged particle at position $z=(0,0,-3/2)\trans$ of charge $\beta$.
The physical constants for specific gravity and electric force are set equal to 1.
We chose this test problem so that it would have chaotic behaviour and conserved energy, with $\LieSO(3)$ as configuration space.

If we let $y_{+}(t), y_{-}(t)$ denote the positions of the positive and negative charge, respectively, the position of the pendulum can be described uniquely by the matrix $g(t) \in \LieSO(3)$ such that $y_{\pm}(t) = g(t)y_{\pm}^0$, where $y_{\pm}^{0} = (0, \pm \alpha, -1)\trans$ is the position of the two charged particles in reference or body coordinates.
Using the standard identification of $\Lieso(3)$ with $\RR^3$ and of $\Lieso(3)^{*}$ with $\RR^3$ via the standard inner product, the state of the system $(g, \mu)$, can be represented with $g\in \LieSO(3)$ as a $3\times 3$ real matrix, and $\mu \in \Lieso(3)^{*}$ as a vector in $\RR^3$.

The right-trivialized Hamiltonian of this system is
\[
\mathcal{H}(g, \mu) = \frac 12 \mu\trans g \II^{-1} g\trans \mu + m \mathrm{e}_3\trans g \mathrm{e}_3 + q \beta \bigl( \bigl\lVert g y_+^0-z\bigr\rVert^{-1} - \bigl\lVert g y_-^0-z\bigr\rVert^{-1} \bigr),
\]
where $\II = m \diag(1+\alpha^2, 1, \alpha^2)$  is the inertia tensor of the pendulum.

\subsection{Order tests}
\label{sec:order-tests}

\begin{table}
\renewcommand{\arraystretch}{1.4}
 \subfloat[Second order Gauss method]{
 \parbox[b]{0.45\textwidth}{
\begin{gather*}
\begin{array}{c|c}
 \frac 12 & \frac 12 \\
 \hline
  & 1
\end{array}
\\
r=0
\end{gather*}
\vskip-\bigskipamount
}
}\quad
 \subfloat[Kutta's third order method]{
 \parbox[b]{0.45\textwidth}{
\begin{gather*}
\begin{array}{c|ccc}
 0  &  0 & 0 & 0 \\
 \frac 12 & \frac 12 & 0 & 0 \\
 1 & -1 & 2 & 0 \\
 \hline 
  & \frac 16 & \frac 23 & \frac 16
\end{array}
\\
r=1
\end{gather*}
\vskip-\bigskipamount
}
}
\\
\subfloat[Fourth order Gauss method]{
\parbox[b]{0.45\textwidth}{
\begin{gather*}
 \begin{array}{c|cc}
 \frac 12-\frac {\sqrt{3}}{6} & \frac 14 & \frac 14 - \frac{\sqrt{3}}{6} \\
 \frac 12+\frac {\sqrt{3}}{6} & \frac 14 + \frac{\sqrt{3}}{6} & \frac 14 \\
 \hline
  & \frac 12 & \frac 12
\end{array}
\\
r=2
\end{gather*}
\vskip-\bigskipamount
}
}\quad
 \subfloat[Sixth order Gauss method]{
 \parbox[b]{0.45\textwidth}{
\begin{gather*}
 \begin{array}{c|ccc}
 \frac 12-\frac {\sqrt{15}}{10} & \frac 5{36}                        & \frac 29 - \frac{\sqrt{15}}{15} & \frac 5{36} - \frac{\sqrt{15}}{30} \\
 \frac 12                        & \frac 5{36} + \frac{\sqrt{15}}{24} & \frac 29                        & \frac 5{36} - \frac{\sqrt{15}}{24} \\
 \frac 12+\frac {\sqrt{15}}{10} & \frac 5{36} + \frac{\sqrt{15}}{30} & \frac 29 + \frac{\sqrt{15}}{15} & \frac 5{36}         \\
 \hline
  & \frac 5{18} & \frac 4{9} & \frac 5{18}
\end{array}
\\
r=4
\end{gather*}
\vskip-\bigskipamount
}
}
\caption{Butcher-tableaux of the RKMK methods tested}
\label{tab:RKMKbutcher}
\end{table}

\begin{table}
\renewcommand{\arraystretch}{1.4}
 \subfloat[Second order midpoint method]{
 \parbox[b]{0.45\textwidth}{
\begin{gather*}
\begin{array}{c|c}
 \frac 12 & \frac 12 \\
 \hline
  & 1
\end{array}
\\
\phantom{\gamma_1 = \frac{1}{2-2^{1/3}}, \quad \gamma_2 = \frac{-2^{1/3}}{2-2^{1/3}}}
\end{gather*}
\vskip-\bigskipamount
}}\quad
 \subfloat[Fourth order DIRK method based on triple jump]{
 \parbox[b]{0.45\textwidth}{
\begin{gather*}
\begin{array}{c|ccc}
 \frac 12 \gamma_1 & \frac 12 \gamma_1  &  0 & 0 \\
 \frac 12 & \gamma_1 & \frac 12 \gamma_2 & 0 \\
 1 - \frac 12 \gamma_1 & \gamma_1 & \gamma_2 & \frac 12 \gamma_1 \\
 \hline 
  & \gamma_1 & \gamma_2 & \gamma_1
\end{array}
\\
\gamma_1 = \frac{1}{2-2^{1/3}}, \quad \gamma_2 = \frac{-2^{1/3}}{2-2^{1/3}}
\end{gather*}
\vskip-\bigskipamount
}}
\\
\subfloat[Sixth order DIRK method]{
\parbox[b]{.95\textwidth}{
\begin{gather*}
 \begin{array}{c|ccccccc}
 \frac 12 \gamma_1 & \frac 12 \gamma_1 & 0 & 0 & 0 & 0 & 0 & 0 \\
 \gamma_1 + \frac 12 \gamma_2 & \gamma_1 & \frac 12 \gamma_2 & 0 & 0 & 0 & 0 & 0 \\
 \gamma_1 + \gamma_2 + \frac 12 \gamma_3 & \gamma_1 & \gamma_2 & \frac 12 \gamma_3 & 0 & 0 & 0 & 0 \\
 \frac 12 & \gamma_1 & \gamma_2 & \gamma_3 & \frac 12 \gamma_4 & 0 & 0 & 0 \\
 1 - (\gamma_1 + \gamma_2 + \frac 12 \gamma_3) & \gamma_1 & \gamma_2 & \gamma_3 & \gamma_4 & \frac 12 \gamma_3 & 0 & 0 \\
 1 - (\gamma_1 + \frac 12 \gamma_2) & \gamma_1 & \gamma_2 & \gamma_3 & \gamma_4 & \gamma_3 & \frac 12 \gamma_2 & 0 \\
 1 - \frac 12 \gamma_1 & \gamma_1 & \gamma_2 & \gamma_3 & \gamma_4 & \gamma_3 & \gamma_2 & \frac 12 \gamma_1 \\
 \hline
 & \gamma_1 & \gamma_2 & \gamma_3 & \gamma_4 & \gamma_3 & \gamma_2 & \gamma_1
\end{array}
\\[\medskipamount]
\begin{aligned}
	\gamma_1 &= 0.78451361047755726381949763, &\quad \gamma_2 &= 0.23557321335935813368479318, \\
	\gamma_3 &= -1.17767998417887100694641568, &\quad \gamma_4 &= 1.31518632068391121888424973
\end{aligned}
\end{gather*}
\vskip-\bigskipamount
}
}
\caption{Butcher-tableaux of the VCG methods tested}
\label{tab:VCGbutcher}
\end{table}
The VRKMK methods that were tested are based on the 1-, 2- and 3-stage Gauss methods, and Kutta's third order method.
These methods are defined by the Butcher tableaux and cut-off parameters in \Fref{tab:RKMKbutcher}.
These methods can be shown to satisfy the extra order conditions for variational integrators to their respective orders.

The order of the VCG methods were also tested.
To obtain higher order, symmetric composition of the midpoint method (which is symmetric, see \Fref{exa:midpoint}) as described in \cite[Section~V.3.2]{hairer06} was used.
The Butcher tableaux of the resulting methods are the same as those of the fourth and sixth order diagonally implicit Runge--Kutta methods (DIRK) shown in \Fref{tab:VCGbutcher}.
The parameters~$\gamma_1,\dotsc, \gamma_4$ were derived by Yoshida \cite{yoshida90}.

The methods were implemented in \textsc{Matlab}, using a modified version of the DiffMan package~\cite{engoe01-1} for defining Lie algebra and Lie group classes and functions on these spaces.
The sets of non-linear equations~\Fref{eq:VRKMKgenq} and \Fref{eq:vcg-integrator} were solved by fixed-point iteration.
The iteration was terminated when the norm of the residual became less than $10^{-11}$.

In these tests we have used the data
\begin{align*}
 m&=q=\beta=1, \\
 \alpha &= 0.1, \\
 g(0) &=\begin{bmatrix} 
         1 & 0 & 0 \\ 
         0 & 0 &-1 \\
         0 & 1 & 0
        \end{bmatrix},\\
 \mu(0) &= g(0) \II g(0)\trans \mathrm{e}_2.
 \end{align*}
The initial data $\mu(0)$ is chosen so that the first component of $f\bigl(g(0), \mu(0)\bigr)$ is $\mathrm{e}_2$.

The errors in $\mu(0.5)$ and $g(0.5)$ with respect to a reference solution are shown in \Fref{fig:order}.
The errors plotted are $\norm{\mu-\mu_{\text{ref}}}_2 +\norm{g-g_{\text{ref}}}_2$, where the first $\norm{\cdot}_2$ is the Euclidean vector norm, and the second is the subordinate matrix norm.
The reference solution was calculated using the sixth order VRKMK method with step size $h=10^{-3}$.
The dashed lines are reference lines for the appropriate orders and are the same lines in the two plots.
As is evident from the plots, errors from fixed point iteration dominates the errors for the 6th order methods when $h$ is small, and the methods appear to obtain their theoretical order.

Analytically, the second order VRKMK and VCG methods are actually identical.
The implementations of the two methods are different, as the non-linear equations are set up in slightly different manners.
The result of this is that the numerical solutions differ slightly.
For this numerical test, the error constants of the VRKMK methods are smaller than those of the VCG methods.

\begin{figure}
\centering
  \subfloat[Absolute error of VRKMK methods]{
  \begin{tikzpicture}[scale=0.8]
  \begin{loglogaxis}[%
     legend pos = south east,
      xlabel ={Step size},
      ylabel ={Error}]
    \addplot table [x = H, y = secondorder]{VRKMKerrortable.dat};
    \addlegendentry{2nd order}
    \addplot table[x = H, y=ERROR]{VRKMk3rdordererrortable.dat};
    \addlegendentry{3rd order}
    \addplot[brown!60!black, every mark/.append style={fill=brown!80!black}, mark=triangle*] table [x = H, y = fourthorder]{VRKMKerrortable.dat};
    \addlegendentry{4th order}
    \addplot table [x = H, y = sixthorder]{VRKMKerrortable.dat};
    \addlegendentry{6th order}
    \addplot[mark=none, dashed, domain=7e-4:0.1 ,samples=5, color=blue] {2*x*x};
    \addplot[mark=none,dashed,domain=7e-4:0.1,samples=5, color=red] {x^3};
    \addplot[mark=none,dashed,domain=7e-4:0.1,samples=5, color=brown!60!black] {x^4};
    \addplot[mark=none,dashed,domain=7e-4:0.1,samples=5, color=black] {0.1*x^6}; 
  \end{loglogaxis}
  \end{tikzpicture}
}
\quad
\subfloat[Absolute error of VCG methods]{
\begin{tikzpicture}[scale=0.8]
\begin{loglogaxis}[%
     legend pos = south east,
      xlabel ={Step size},
      ylabel ={Error}]
    \addplot table [x = H, y = secondorder]{CGorder.dat};
    \addlegendentry{2nd order}
    \addplot[brown!60!black, every mark/.append style={fill=brown!80!black}, mark=triangle*] table [x = H, y = fourthorder]{CGorder.dat};
    \addlegendentry{4th order}
    \addplot[black, mark=star] table [x = H, y = sixthorder]{CGorder.dat};
    \addlegendentry{6th order}
    \addplot[mark=none, dashed, domain=7e-4:0.1 ,samples=5, color=blue] {2*x*x};
    \addplot[mark=none, dashed, domain=7e-4:0.1,samples=5, color=brown!60!black] {x^4};
    \addplot[mark=none, dashed, domain=7e-4:0.1,samples=5, color=black] {0.1*x^6};
  \end{loglogaxis}
\end{tikzpicture}
}
\caption{Order plot. Dashed lines are reference lines for the appropriate orders}
\label{fig:order}
\end{figure}
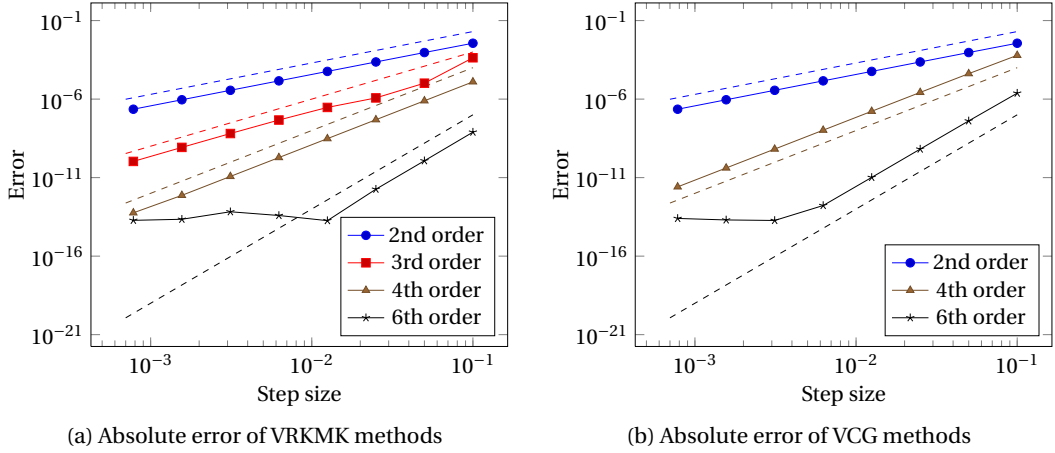

\subsection{Long time behaviour}

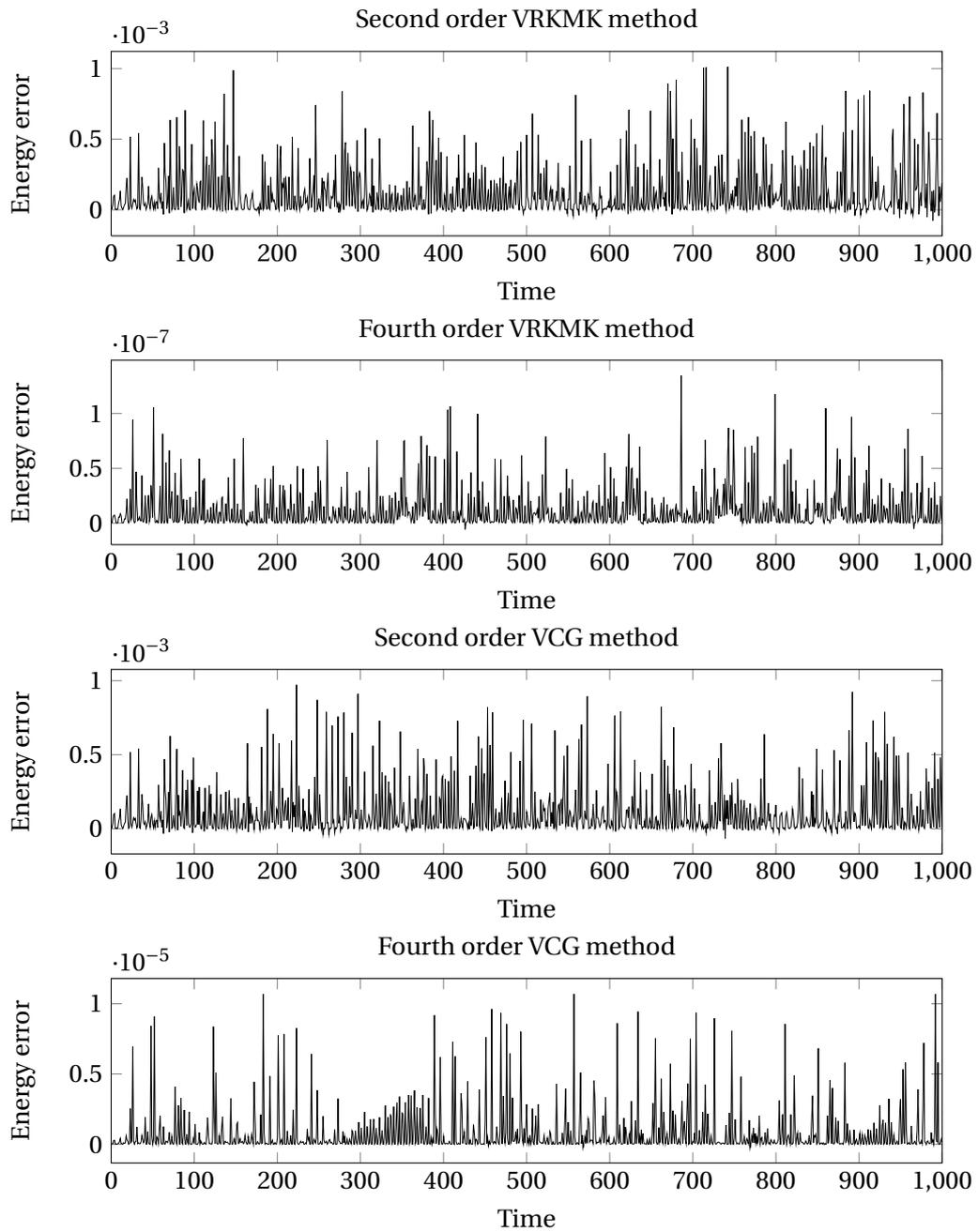
\begin{figure}
\centering
\begin{tikzpicture}
  \begin{axis}[
      width = 0.9\textwidth,
      height = 0.2\textheight,
      xlabel = {Time},
      ylabel = {Energy error},
      xmin=0, xmax=1000,
      title = {Second order VRKMK method}]
    \addplot[very thin, mark=none] table [x = T, y = Herr]  {VRKMKlongrun2skipped.dat};
  \end{axis}
\end{tikzpicture}
\begin{tikzpicture}
  \begin{axis}[
      width = 0.9\textwidth,
      height = 0.2\textheight,
      xlabel = {Time},
      ylabel = {Energy error},
      xmin=0, xmax=1000,
      title = {Fourth order VRKMK method}]
    \addplot[very thin, mark=none] table [x = T, y = Herr]  {VRKMKlongrun4skipped.dat};
  \end{axis}
\end{tikzpicture}
\begin{tikzpicture}
  \begin{axis}[
      width = 0.9\textwidth,
      height = 0.2\textheight,
      xlabel = {Time},
      ylabel = {Energy error},
      xmin=0, xmax=1000,
      title = {Second order VCG method}]
    \addplot[very thin, mark=none] table [x = T, y = Herr]  {VCGlongrun2skipped.dat};
  \end{axis}
\end{tikzpicture}
\begin{tikzpicture}
  \begin{axis}[
      width = 0.9\textwidth,
      height = 0.2\textheight,
      xlabel = {Time},
      ylabel = {Energy error},
      xmin=0, xmax=1000,
      title = {Fourth order VCG method}]
    \addplot[very thin, mark=none] table [x = T, y = Herr]  {VCGlongrun4skipped.dat};
  \end{axis}
\end{tikzpicture}
\caption{Energy error}
\label{fig:longrun}
\end{figure}

The long term behaviour of the methods was also investigated.
In \Fref{fig:longrun}, the energy error of the numerical solution is plotted over the time span $(0, 1000)$.
We have used step size $h=0.01$ ($10^5$ integration steps).
Only the second and fourth order methods were tested on this time span.
As can be seen from the plots, the energy error is small, approximately $10^{-3}$ for both second order methods, and approximately $10^{-7}$ for the fourth order VRKMK method and about $10^{-5}$ for the fourth order VCG method.


\section{Future work}
The reformulation of RKMK methods with a Pad\'{e} approximation of $\dexp^{-1}$ is briefly discussed in \Fref{sec:pade}.
This reformulation makes it possible to eliminate the Lagrange multipliers~$\lambda_i$, which is beneficial for computational efficiency.
We expect that the proof of the order of VRKMK methods, \Fref{thm:orderRKMK}, will carry over to these methods without complications.
Implementation and study of this approach would make the variational RKMK methods more competitive in terms of computational cost. 

Another class of Lie group methods is formed by the commutator-free Lie group methods.
The approach described in this article can easily be used to formulate symplectic Lie group methods based on commutator-free methods.
Formulation, implementation and study of variational commutator-free methods are aspects that can be pursued in the future.

A desirable result would be generalization of these integrators to homogenous spaces.
This has proven to be more difficult than one could hope.
In general, the problem arises due to isotropy.
If $M$ is a homogeneous $G$-space with $\dim(M)<\dim(G)$,
then the infinitesimal action at a point $z\in M$, 
\[\mathfrak{g} \ni \xi \mapsto \frac{\partial}{\partial t}\exp(t\xi)\cdot z \in T_{z}M,\]
is not injective.
Therefore, to identify a vector in $T_zM$ with some element in $\mathfrak{g}$, a choice has to be made.

The main idea of variational integration is to minimize the action.
Inspired by this, one could attempt the following approach, sketched out for a variational method based on the one-stage $\theta$-method for $0\le\theta\le1$.
Assume the action is from the left, and denote the action as $g\cdot q$, and the infinitesimal action as $\xi\cdot q$.
Let $\ell(q,\xi)=L(q, \xi\cdot q)$ be the ``trivialised'' Lagrangian, and use the discrete Lagrangian
\begin{equation}L_h(q_0, q_1)=\min_{\xi}h\ell\bigl(\exp(h\theta \xi)\cdot q_0,\xi\bigr)
\label{eq:homspaceaction}
\end{equation}
where the minimum is taken over all $\xi$ such that $\exp(h\xi) \cdot q_0=q_1$.
If the minimizing equation can be solved, this discrete Lagrangian can be used to construct symplectic integrators.
However, the following example shows that in some cases, the minimizing equation has no solution.
Let $M=\RR$ and the group action that of affine functions $\RR\to\RR$, i.e., for $(a,b)\in (\RR \smallsetminus \{0\})\times \RR=G$, 
$(a,b)\cdot q=aq+b$.
Let the Lagrangian be that of a free particle, $L(q,\dot{q})= \frac{1}{2}\dot{q}^2$.
In this case it turns out that the minimizing problem \Fref{eq:homspaceaction} can be expressed as an unconstrained one-dimensional problem,
\[L_h(q_0,q_1)= \min_{x\in \RR} \frac{h}{2}\biggl(\frac{x \ee^{h\theta x}}{\ee^{hx}-1}\biggr)^2(q_1-q_0)^2.\]
However, this has no solution if $q_1 \ne q_0$, so $L_h(q_0,q_1)$ is not defined.
Furthermore, for $0<\theta<1$, the expression has a \emph{maximizer}, so a naive solution to the extremization problem would return the maximizing solution.
The symplectic method based on such a solution is not even consistent.

\section{Conclusion}
In this article, a set of equations defining symplectic integrators for ODEs on $\coT G$ were presented, as well as two classes of integrators using these equations.
The integrators obtained are formulated intrinsically on $\coT G$, and any drift away from the manifold in numerical solutions is due to round-off errors.
The integrators were developed as variational methods for Lagrangian problems, and are therefore symplectic when applied to Hamiltonian differential equations.
Both classes that were studied, were shown to contain methods of arbitrarily high order, although the computational cost per time step increases with the order.
Effective implementation of the methods has not been a major goal in this article, we have instead focused on the properties of these methods.

The two classes of symplectic methods are based on, respectively, the Runge--Kutta--Munthe-Kaas methods, and the Crouch--Grossman methods.
The methods have a partitioned structure where the position on the Lie group is integrated by the Lie group method while the momentum is integrated by formulae which involve various functions on $\g^{*}$.
We can therefore say that these methods are partitioned Lie group methods and are Lie group methods in a wide understanding of that term.
To the knowledge of the authors, this is the first time that symplectic Lie group methods have been presented and studied in the level of detail done in this article.

\section*{Acknowledgements}

We would like to thank our supervisor, Brynjulf Owren, for encouragement and many helpful discussions.
We would also like to thank Klas Modin for fruitful discussions leading to the writing of this paper.
Finally, we would like to thank the two anonymous referees for helpful comments and suggestions.
The research was supported by the Research Council of Norway, and by a Marie Curie International Research Staff Exchange Scheme Fellowship within the 7th European Community Framework Programme.

\printbibliography

\end{document}